\documentclass[11pt,leqno]{amsart}

\usepackage[top=2.5 cm,bottom=2 cm,left=2.5 cm,right=2 cm]{geometry}
\usepackage[utf8]{inputenc}
\usepackage{xfrac}
\usepackage{bbm}
\usepackage{esint}
\usepackage{graphicx}
\usepackage{hyperref}
\hypersetup{
    hidelinks,
    }
\usepackage{color}
\usepackage{amssymb}
\usepackage{amsfonts}
\usepackage{psfrag}
\newcounter{stepnb}

\usepackage{tikz}
\usetikzlibrary{shapes.geometric,calc,decorations.markings,math}

\tikzstyle{nodo}=[circle,draw,fill,inner sep=0pt,minimum size=%
1.5mm]
\tikzstyle{infinito}=[circle,inner sep=0pt,minimum size=0mm]

\newtheorem{theorem}{Theorem}[section]
\newtheorem{lemma}[theorem]{Lemma}
\newtheorem{proposition}[theorem]{Proposition}
\newtheorem{corol}[theorem]{Corollary}

\newtheorem{remark}[theorem]{Remark}

\numberwithin{equation}{section}

\newcommand{\R}{\mathbb{R}}

\DeclareMathOperator*{\essinf}{ess\,inf}
\DeclareMathOperator*{\essup}{ess\,sup}

\newcommand{\ee}{\varepsilon}

\newcommand{\be}{\begin{equation}}
\newcommand{\eq}{\end{equation}}

\newcommand{\weaks}{\stackrel{*}{\rightharpoonup}}
\newcommand{\loc}{\mathrm{loc}}

\begin{document}
\title[]{Nonlocal Generalized Aw-Rascle-Zhang model: well-posedness and singular limit}

\author[E. Marconi]{Elio Marconi}
\address{E.M. Dipartimento di Matematica ``Tullio Levi Civita", Universit\`a degli Studi di Padova, Via
Trieste 63, 35131 Padova, Italy}
\email{elio.marconi@unipd.it}
\author[L.~V.~Spinolo]{Laura V.~Spinolo}
\address{L.V.S. CNR-IMATI ``E. Magenes", via Ferrata 5, I-27100 Pavia, Italy.}
\email{spinolo@imati.cnr.it}
\maketitle
{
\rightskip .85 cm
\leftskip .85 cm
\parindent 0 pt
\begin{footnotesize}
We discuss a nonlocal version of the Generalized Aw-Rascle-Zhang model, a second-order vehicular traffic model where the empty road velocity is a Lagrangian marker governed by a transport equation. The evolution of the car density is described by a continuity equation where the drivers' velocity depends on both the empty road velocity and the convolution of the car density with an anisotropic kernel.  We establish existence and uniqueness results. When the convolution kernel is replaced by a Dirac Delta, the nonlocal model formally boils down to the classical (local) Generalized Aw-Rascle-Zhang model, which consists of a conservation law coupled with a transport equation. In the case of exponential kernels, we establish convergence in the nonlocal-to-local limit by proving an Oleinik-type estimate for the convolution term. To the best of our knowledge, this is the first nonlocal-to-local limit result for a system of two non-decoupling equations with a nonlocal flux function.

\medskip



\end{footnotesize}
}

\section{Introduction}
We consider the system 
\be \label{e:nlGARZ}
\left\{
\begin{array}{ll}
        \partial_t \rho + \partial_x [V(\xi, u) \rho] =0 \\
        \partial_t u + V(\xi, u) \partial_x u = 0, \\
\end{array} 
\right.  \quad  \xi (t, x) = \int_x^{+ \infty} \rho(t, y) \eta(x-y) dy. 
\eq
In the above expression, the unknowns are $\rho: \R_+ \times \R \to \R$ and $u: \R_+ \times \R \to \R$, whereas the velocity field $V: \R^2 \to \R$ and the the convolution kernel $\eta: \R \to \R$ are given functions. The above system is a nonlocal version of the so-called Generalized Aw-Rascle-Zhang vehicular model introduced by Fan, Herty, and Seibold in~\cite{FanHertySeibold}, namely 
\be \label{e:GARZ}
\left\{
\begin{array}{ll}
        \partial_t \rho + \partial_x [V(\rho, u) \rho] =0 \\
        \partial_t u + V(\rho, u) \partial_x u = 0. \\
\end{array}
\right.
\eq
In both~\eqref{e:nlGARZ} and ~\eqref{e:GARZ}, $\rho$ represents the car density and $V$ their velocity, whereas the Lagrangian marker $u$ represents the empty road velocity, that is, the velocity each driver would choose if the road were completely free. This is an individual feature of each driver, and, as such, it is transported by the flow as dictated by the second equation in~\eqref{e:nlGARZ} and~\eqref{e:GARZ}. The first equation in both systems expresses the conservation of the total amount of cars. We point out in passing that the second equation of both~\eqref{e:nlGARZ} and~\eqref{e:GARZ} is in principle highly ill-posed in low regularity regimes of $\rho$ and $u$ as the product $V(\rho, u) \partial_x u$ is not well-defined: we will come back to this point in the following. 

As in~\cite{FanHertySeibold}, we impose on $V$ the following assumptions, which are fairly reasonable in view of modeling considerations\footnote{{\color{black}To simplify the notation, in~\eqref{e:V} we assume that all the conditions are satisfies on the whole $\R^2$, but since as a matter of fact in the following we will show that the solutions $(\xi,u)$ of \eqref{e:nlGARZ} and $(\rho,u)$ of \eqref{e:GARZ} are always confined in the rectangle $[0, 1] \times [\essinf u_0, \essup u_0]$, it suffices that the conditions are satisfied in that range.}}: 
   \be \label{e:V}
         V\in C^2 (\R^2; \R), \quad V \ge 0, \quad 
    \partial_1 V \leq 0, \quad  \partial_2 V \ge 0, \quad V(1, w) =0 \; \text{for every $w$},
\eq 
where $\partial_1V$ and $\partial_2V$ denote the partial derivatives of $V$ with respect to $\rho$ and $u$, respectively. In the above equation and in the rest of the paper, we normalize to $1$ the maximal possible car density, that is the density at which all drivers stop. 

In~\eqref{e:nlGARZ}, the presence of the convolution term takes into account the fact that traffic agents tune their velocity based on the averaged density in a neighborhood of their position, rather than on its pointwise value only as in~\eqref{e:nlGARZ}. As it is by now fairly standard in the analysis of nonlocal traffic models (see for instance~\cite{BlandinGoatin,Chiarello}), we assume that 
\be 
\label{e:eta}
   \eta \in L^1 \cap L^\infty(\R; \R_+), \quad \int_{\R} \eta (x) dx  = 1, \quad \mathrm{supp} \,\eta \subseteq \R_-, \quad \eta \,   \text{monotone non-decreasing on $\R_-$}. 
\eq 
The third assumption in the above equation is a ``look-ahead-only" condition modeling the fact that drivers are affected by the downstream traffic density only. The last condition in~\eqref{e:eta} models the fact that drives pay more attention to closer vehicles rather than those that are further away. In the following, we focus on the Cauchy problem posed by prescribing the initial data
\be \label{e:idpose}
     \rho(0, \cdot) = \rho_0, \qquad u(0, \cdot) = u_0,
\eq
which in view of modeling considerations we assume satisfy 
\begin{equation} \label{e:id}
0 \leq \rho_0 \leq 1, \quad u_0 \in L^\infty (\R), \; u_0 \ge 0. 
\end{equation}
In the present paper we establish well-posedness of~\eqref{e:nlGARZ},\eqref{e:idpose} in a suitable class of solutions. We also show that, under suitable assumptions, the solution of~\eqref{e:nlGARZ} approaches the solution of~\eqref{e:GARZ} when $\eta$ converges to the Dirac Delta.  Well-posedness results for~\eqref{e:GARZ},\eqref{e:idpose} are established in the companion note~\cite{MS:localGARZ}.

The archetypical fluid-dynamic traffic model is the so-called LWR (Lighthill, Whitham, and Richards) model introduced in \cite{LW,R}, namely  
\be \label{e:LWR}
\partial_t \rho + \partial_x [V(\rho) \rho] =0.
\eq 
Despite its numerous advantages and its wide use in the engineering literature, the LWR model is affected by some shortcomings. In particular, the model postulates that every driver reacts in the same way to the traffic conditions, which is not always the case in real-world applications. To overcome these drawbacks, various works introduced higher-order models, that is, models consisting of several equations, see in particular~\cite{AwRascle,Daganzo,Payne,Whitham,Zhang} and also~\cite{GaravelloHanPiccoli} for an extended overview. In particular, the model~\eqref{e:GARZ}, introduced in~\cite{FanHertySeibold}, is a second-order model where the way each driver reacts to the traffic evolution is governed by his/her empty road velocity $u$. 

Another  nonlocal version of~\eqref{e:nlGARZ} was introduced in~\cite{CFGG}, where the authors obtained the system 
\be \label{e:nlGARZ2}
\left\{
\begin{array}{ll}
        \partial_t \rho + \partial_x \Big[[V(\rho, u) \ast \eta] \rho\Big] =0 \\
        \partial_t u + [V(\rho, u) \ast \eta] \partial_x u = 0, \\
\end{array} 
\right.  
\eq
as the result of a micro-macro limit argument. The above model, and its nonlocal-to-local limit towards~\eqref{e:GARZ}, is currently investigated in~\cite{ACCGK}. Compared to~\eqref{e:nlGARZ}, a considerable advantage of~\eqref{e:nlGARZ2} is that in~\eqref{e:nlGARZ2} the smoothing effect of the convolution acts on \emph{both} $\rho$ and $u$, which is a remarkable boon to its analytic treatment. On the other hand, the fact that in~\eqref{e:nlGARZ} the velocity function $V$ is only affected by the convolution term through the perceived density might be advantageous from the modeling viewpoint, as in this way $V$ depends on the exact value of the Lagrangian marker $u$ as in~\eqref{e:nlGARZ}, rather than on a smoothed version as in~\eqref{e:nlGARZ2}. Other traffic models involving \emph{systems} of nonlocal equations are discussed for instance in~\cite{AndreianovSylla} and~\cite{CPSZ}.   

We now establish the well-posedness of the Cauchy problem~\eqref{e:nlGARZ},\eqref{e:idpose}. 
\begin{theorem}
\label{t:wpnl}
Assume~\eqref{e:V},\eqref{e:eta},~\eqref{e:id} and $\rho_0 \in L^1(\R)$.  Assume furthermore that
\be \label{e:id2}
    u_0 \in W^{1 \infty} (\R), \quad  u_0' = z_0 \rho_0, \quad z_0 \in L^\infty (\R). 
\eq
Then there is $T= C(V,\| z_0 \|_{L^\infty})$ (that is, $T$ only depends on the velocity function $V$ and on $\| z_0 \|_{L^\infty}$) such that there is exactly one solution of~\eqref{e:nlGARZ},\eqref{e:idpose} in the class $(\rho, u) \in 
L^\infty_{\mathrm{loc}} ([0, T[; L^\infty(\R_+))\times L^\infty_{\mathrm{loc}}([0, T[; W^{1 \infty} (\R)). $ Also, the unique solution satisfies the bound
\be \label{e:localex}
      0 \leq \rho(t, x) \leq \frac{1}{1- t C(V) \| z_0 \|_{L^\infty}} \quad \text{a.e. $(t, x) \in ]0, T[ \times \R$},
\eq
the two-sided estimate
\be 
\label{e:mpxi}
     0 \leq \xi(t, x) \leq 1 \; \text{for every $(t, x) \in \R_+ \times \R$ }
\eq
and also
\be \label{e:mpii}
    \inf u_0 \leq u \leq \sup u_0, \quad  \partial_x u = \rho z, \quad \| z \|_{L^\infty} \leq  \| z_0 \|_{L^\infty}. 
\eq
We also have the following propagation of regularity results:
if we further assume
\be \label{e:idgarz}
      z_0' = \rho_0 \psi_0, \quad \text{for some $\psi_0 \in L^\infty (\R)$}
\eq
then
\begin{equation}\label{e:est_h}
\partial_x z  = \psi \rho, \; \text{with} \;  \|\psi \|_{L^\infty}\le \|\psi_0\|_{L^\infty};
\end{equation}
if moreover $ \rho_0 \in W^{1\infty}(\R)$, then 
\begin{equation}\label{e:est_rho}
 \rho\in W^{1\infty}_{\mathrm{loc}}([0,T[\times \R) \qquad \mbox{and}\qquad \partial_x u \in W^{1\infty}_{\mathrm{loc}}([0,T[\times \R).
\end{equation} 
\end{theorem}
Some remarks are here in order. First, the only reason why we only have local-in-time well-posedness is because we cannot rule out a finite time blow-up of the $L^\infty$ norm of the density $\rho$. However, it is fairly easy to see that, under some further assumptions, one can replace~\eqref{e:localex} with a uniform-in-time control, which yields global-in-time well-posedness, see Remark~\ref{r:mp} or Lemma~\ref{l:gte} in \S\ref{s:nl}. Second, an interesting point of Theorem~\ref{t:wpnl} is that the car density $\rho$, albeit always nonnegative, can exceed the value $1$, whereas the \emph{convoluted density} $\xi$, the one on which the velocity depends,  satisfies~\eqref{e:mpxi}. It would be interesting to assess whether this is just a drawback of the nonlocal GARZ system~\eqref{e:nlGARZ} or has some modeling interpretation. A first very tentative guess could be for instance to postulate that, owing to road safety rules, $\rho=1$, the density at which drivers stop, does not exactly correspond to bumper-to-bumper packing, but rather to a lower density which could be exceeded in some exceptional\footnote{Note that the fact that the convoluted density $\xi$ is bounded by $1$ imposes constraints on the measure of the set where $\rho$ exceeds $1$} situations. Third, as mentioned before, a drawback of~\eqref{e:nlGARZ} compared to~\eqref{e:nlGARZ2} is that the regularizing effect of the convolution only acts on $\rho$. This means that, if $u$ has low regularity, so in general does $V(\cdot, u)$, and this in turn implies that the classical method of characteristics, which is widely used~\cite{CLM,KeimerPflug} in the study of nonlocal conservation laws, does not apply (at least, not in its elementary form requiring Lipschitz regularity of $V$). To circumvent this obstruction, in Theorem~\ref{t:wpnl} we impose Lipschitz-type regularity on $u_0$, and then show (this being the nonobvious point) that this is propagated by the solution semigroup. The regularity assumption imposed on $u$ in~\eqref{e:id2} is actually stronger than Lipschitz continuity, as we are also loosely speaking imposing that $u$ is constant on the sets where $\rho$ vanishes. Note, however, that discussing the value of $u$ (the drivers' empty road velocity) on sets where $\rho$ (the drivers' density) vanishes is a sort of mathematical artifact, and hence this further assumption is fairly innocuous from the modeling viewpoint.  Fourth, we point out that, since $u$ belongs to $ L^\infty_{\mathrm{loc}}([0, T[; W^{1 \infty} (\R))$, the product $V(\xi, u)\partial_x u$ is well-defined and there is no need to provide a distributional formulation of the equation at the second line of~\eqref{e:nlGARZ}, which holds as an identity between $L^\infty$ functions. Finally, the solutions obtained in Theorem~\ref{t:wpnl} are stable with respect to perturbations in the initial data, see Corollary~\ref{c:stability}.

We now discuss the singular nonlocal-to-local limit: we fix a parameter $\ee >0$, and consider the Cauchy problem
\be \label{e:nlGARZee}
\left\{
\begin{array}{ll}
        \partial_t \rho_\ee + \partial_x [V(\xi_\ee, u_\ee) \rho_\ee] =0 \\
        \partial_t u_\ee + V(\xi_\ee, u_\ee) \partial_x u_\ee = 0 \\
\end{array}
\right. \quad \text{where} \quad \xi_\ee (t, x) : = \frac{1}{\ee}\int_x^{+\infty}  \! \! \! \! \eta \left( \frac{x-y}{\ee} \right) \rho_\ee (t, y) dy,
\eq
which is nothing else but~\eqref{e:nlGARZ} with $\eta$ replaced by $\eta_\ee (\cdot): = \eta(\sfrac{\cdot}{\ee})/\ee$. In the vanishing $\ee$ limit, $\eta_\ee$ converges weakly$^\ast$ in the sense of measures to the Dirac Delta, and hence~\eqref{e:nlGARZee} formally boils down to~\eqref{e:GARZ}. In recent years, several works have been devoted to the analysis of the singular nonlocal-to-local limit in the case of the \emph{scalar} LWR model~\eqref{e:LWR}. Despite the existence of counter-examples ruling out a general convergence result~\cite{ColomboCrippaSpinolo}, in the specific framework of traffic models, with anisotropic convolution kernels, one can establish convergence under fairly mild assumptions, see~\cite{BressanShen, BressanShen2,8autori,CocliteCoronDNKeimerPflug,ColomboCrippaMarconiSpinolo,CCMS,KeimerPflug2}. The analysis of the nonlocal-to-local limit in the case of systems is much less understood. Indeed, we are only aware of two works: \cite{ChiarelloKeimer} focuses on a system of balance laws where the nonlocality is confined in the source term, whereas the system studied in~\cite{CDN} consists of a scalar nonlocal conservation law coupled with a sort of continuity equation. To the best of our knowledge, the following is the first nonlocal-to-local limit result for a system of two non-decoupling equations with a nonlocal flux.
\begin{theorem}
\label{t:oleinik}
Assume that $\rho_\ee, u_\ee, \xi_\ee$ are the same as in~\eqref{e:nlGARZee}, and that $\eta(x)= e^x \mathbbm{1}_{\R_-}$. Let the initial data $\rho_0 \in BV(\R)$ and $u_0 \in W^{1\infty}(\R)$ satisfy \eqref{e:id}, \eqref{e:id2}, \eqref{e:idgarz}, and let $V$ satisfy~\eqref{e:V} and the following conditions:
\begin{align}
    \partial_1 V(\rho, u) \leq -\alpha_{V} & < 0, \quad \text{for every $\rho \in [0, 1],  u \in [\inf u_0, \sup u_0]$, } \label{e:V1}\\
    - \beta_V \leq  \partial_{11} V(\rho, u) & \leq  0, \quad \text{for every $\rho \in [0, 1],  u \in [\inf u_0, \sup u_0]$, }
    \label{e:V2}\\
     \alpha^2_V  & > 54 \beta^2_V. \label{e:V3} 
\end{align}
Then there is $\ee_0$ only depending on $V$, $\| z_0 \|_{L^\infty}$ and $\| \psi_0 \|_{L^\infty}$ such that for every $\ee \in ]0, \ee_0[$ the following holds: the solution $(\rho_\ee, u_\ee)$ is defined globally in time and satisfies 
\be \label{e:boundrho}  
    \| \rho(t, \cdot) \|_{L^\infty} \leq 2, \quad \text{for every $t\ge 0$}.
\eq 
Also, $\|\rho_\ee\|_{L^\infty(\R^+\times \R)}\to 1$ and, for every $t>0$, we have 
\[
\begin{split}
 \rho_\ee(t,\cdot) \rightharpoonup^* \rho(t,\cdot) \quad \mbox{weakly$^\ast$ in } L^\infty(\R), & \quad \xi_\ee(t,\cdot) \to \rho(t,\cdot) \quad \mbox{strongly in } L^1_\loc(\R),\\
& u_\ee(t,\cdot) \to u(t,\cdot) \quad \mbox{uniformly in }C^0(\R) \qquad \mbox{as }\ee \to 0^+,
\end{split}
\]
where $(\rho,u)$ is the solution of the Cauchy problem \eqref{e:GARZ},\eqref{e:idpose} such that $\rho$ is an entropy admissible solution of the conservation law at the first line of~\eqref{e:GARZ}.  
\end{theorem}
Note that~\cite[Theorem 1.1]{MS:localGARZ} states that, under the same assumptions as in the statement of Theorem~\ref{t:oleinik}, the Cauchy problem \eqref{e:GARZ},\eqref{e:idpose} has a unique solution belonging to a suitable regularity framework and such that $\rho$ is an entropy admissible solution  of the conservation law. Since, under the hypotheses of Theorem~\ref{t:oleinik}, any accumulation point of $(\rho_\ee, u_\ee)$ belongs to the right regularity framework, this accumulation point is unique. As a matter of fact, the only reason why we require that $\rho_0 \in BV(\R)$ in the statement of Theorem~\ref{t:oleinik} is to make sure that~\eqref{e:GARZ} has a unique solution. If we were only interested in compactness (and hence, convergence up to subsequences), we could drop this requirement. 

In the statement of Theorem~\ref{t:oleinik}, the assumption that $\eta$ is the exponential function is obviously highly restrictive, but was imposed in several previous works on the nonlocal-to-local limit, most notably~\cite{BressanShen,ChiarelloKeimer,8autori,CocliteCoronDNKeimerPflug}. From the technical viewpoint, its main advance is entailing the algebraic identity $\rho_\ee= \xi_\ee - \ee \partial_x \xi_\ee$, which is a boon to the analysis. The proof of Theorem~\ref{t:oleinik} relies on Proposition~\ref{p:oleinik}, which establishes an Oleinik-type estimate for the convolution term $\xi_\ee$. This in turn yields strong compactness in $L^1_{\loc}$ through a uniform-in-$\ee$ control on the total variation. Note that Proposition~\ref{p:oleinik} implies that  the conservation law at the first line of~\eqref{e:GARZ} satisfies the classical Oleinik estimate~\cite{Dafermos:book,Oleinik}, and this is consistent with the fact that assumptions \eqref{e:V1} and \eqref{e:V2} imply that $\rho \mapsto \rho V(\rho,u)$ is strictly concave. Note furthermore that, if the function $u$ is constant, the nonlocal system~\eqref{e:nlGARZ} boils down to the nonlocal LWR model. This in particular implies that we can refer to the counter-example given by~\cite[Theorem 6]{ColomboCrippaMarconiSpinolo} and conclude that in general the total variation of the solution $\rho_\ee(t, \cdot)$ blows up in the vanishing $\ee$ limit, for every $t>0$. This is the reason why in the proof of Theorem~\ref{t:oleinik} we establish a uniform bound on the total variation of the convolution term $\xi_\ee$, rather than of the solution $\rho_\ee$: this is idea was first introduced in~\cite{CocliteCoronDNKeimerPflug} and also used in~\cite{CCMS}. 
\subsection*{Outline} 
The exposition is organized as follows. In~\S\ref{s:nl} we establish Theorem~\ref{t:wpnl}, Lemma~\ref{l:gte} (a global-in-time existence results) and Corollary~\ref{c:stability} (a stability result). In \S\ref{s:olenik} we give the proof of Theorem~\ref{t:oleinik}. For the reader's convenience we conclude the introduction by collecting the main notation used in the paper. 
\subsection*{Notation}
We denote by $C(a_1, \dots, a_n)$ any constant only depending on the quantities $a_1, \dots, a_n$. Its precise value can vary form occurrence to occurrence. We also use the the shorthand notation
\be \begin{array}{ll}\label{e:civ}
    C(V) = C \Big( \max \{  |V(\xi, u)|+ | \nabla V(\xi, u)| + |D^2 V(\xi,u)| : 
     \; \xi \in [0, 1], \; |u| \leq \| u_0 \|_{L^\infty} \} \Big)  .
\end{array}
\eq
\subsubsection*{General mathematical symbols}
\begin{itemize}
\item $\mathbbm{1}_E$: the characteristic function of the set $E$, the one attaining the value $1$ on $E$ and $0$ elsewhere. 
\item $L^1(\Omega), L^\infty(\Omega)$: the Lebesgue spaces of summable and essentially bounded functions, respectively, defined on the measurable set $\Omega \subseteq \R^d$;
\item $W^{1 \infty}(\Omega)$ the Sobolev space of bounded and Lipschitz continuous functions defined on $\Omega$;
\item $BV (\Omega)$: the Banach space of summable functions with bounded total variation defined on $\Omega$; 
\item $X_{\loc} (\Omega)$: the space of functions $f$ defined on $\Omega$ such that $f \in X(K)$ for every $K$ compact set contained in $\Omega$;
\item $C^0 (\Omega)$: the Banach space of continuous and bounded functions defined on the open set $\Omega$;
\item $C^1(\Omega), C^2(\Omega)$: the space of bounded, continuously differentiable and twice continuously differentiable functions, respectively, defined on the open set $\Omega$; 
\item $C^\infty (\Omega)$: the space of infinitely differentiable functions defined on the open set $\Omega$; 
\item  $f\# \mu$:  the push-forward of the measure $\mu$ through the measurable map $f: X \to Y$, namely 
\be \label{e:pushfdef}
   f \# \mu (E)   : =  \mu (f^{-1}(E)), \quad \text{for every $E \subseteq Y$ measurable}
\eq
\end{itemize}
\subsubsection*{Symbols introduced in the present paper}
\begin{itemize}
\item $\rho$: the car density, see~\eqref{e:nlGARZ};
\item $u$: the empty road velocity, see~\eqref{e:nlGARZ};
\item $z$: the same function as in~\eqref{e:mpii};
\item $\xi$: the convolution term, see~\eqref{e:nlGARZ};
\item $\psi$: the same function as in~\eqref{e:est_h};
\item $\alpha_V$, $\beta_V$: see~\eqref{e:V1} and~\eqref{e:V2}, respectively 
\item $h: = \partial_x \xi$, see the first lines of the proof of Proposition~\ref{p:oleinik}.  
\end{itemize}
\begin{remark}\label{r:dafermos}
         We recall~\cite[Lemma 1.3.3]{Dafermos:book} and conclude that for every $T>0$ and every bounded distributional solution $\rho \in L^\infty ([0, T] \times \R)$ of the equation at the first line of ~\eqref{e:nlGARZ}  the following holds. In the equivalence class of $\rho$ (which is an element of a Lebsgue space and, as such, only defined up to negligible sets) there is a representative such that the map $t \to \rho(t, \cdot)$ is continuous from $\R_+$ to $L^\infty (\R)$ endowed with the weak$^\ast$ topology.  In the present paper we always tacitly identify $\rho$ with this weak$^\ast$ continuous representative and discuss the values of $\rho(t, \cdot)$ at \emph{any} $t>0$. 
\end{remark}

\section{Well-posedness of the Cauchy problem~\eqref{e:nlGARZ},\eqref{e:idpose}}\label{s:nl}
This section is mainly devoted to the proof of Theorem~\ref{t:wpnl}. 
We first establish existence by relying on an approximation argument, 
organized as follows: in \S\ref{sss:app} we construct the approximating sequence, in \S\ref{sss:limite} and \S\ref{sss:pconvu2} we pass to the limit and in \S\ref{ss:proofl21} we establish the proof of Lemma~\ref{l:quenne}. Next, in \S\ref{ss:uninl} we establish the uniqueness part of Theorem~\ref{t:wpnl}, in \S\ref{ss:gte} we prove Lemma~\ref{l:gte}, which provides a condition yielding global-in-time existence, and in~\S\ref{ss:stab} we establish Corollary~\ref{c:stability}, that is stability with respect to the initial data.
\subsection{Approximation argument} \label{sss:app}
We proceed according to the following steps. \\
{\bf Step 1:} we define the approximating sequence $\{ \rho_{n} \}_{n \in \mathbb N}$. 
First, we define a sequence $\{ \rho_{0n} \}_{n \in \mathbb N}$ such that
\be \label{e:rho0enne}
   \rho_{0n}: \R \to [0, 1], \;  \rho_{0n} \in  C^\infty_c (\R), \; \| \rho_{0n} \|_{C^0} \leq \| \rho_0 \|_{L^\infty}, \; 
   \| \rho_{0n} \|_{L^1} \leq \| \rho_0 \|_{L^1}, \; \rho_{0n} \to \rho_0 \; \text{strongly in $L^1(\R)$}.
\eq
We also fix a sequence of continuously differentiable functions $z_{0n}$ with $\|z_{0n}\|_{L^\infty}\le \|z_{0}\|_{L^\infty}$ and converging pointwise a.e. to $z_0$.
Next, we point out that $u_0$ is a bounded variation function owing to the identity $u_0'= \rho z_0$ and to the $L^1$ and $L^\infty$ bound on $\rho_0$ and $z_0$, respectively. 
We set 
\be \label{e:uinfty}
     u_\infty: = \lim_{x \to - \infty} u_0(x)
\eq
and 
\be \label{e:koala}
    u_{0n} (x) : = u_\infty + \int_{- \infty}^x \rho_{0n} (t, x) z_{0n} (x) dx  
\eq
and observe that $u_{0n}\in W^{2\infty}(\R)$ and $u_{0n} \to u_0$ uniformly in $C^0(\R)$. Finally, we fix a sequence $\{ \eta_{n} \}_{n \in \mathbb N}$ enjoying \eqref{e:eta} and such that
\be \label{e:etaenne}
    \eta_n \in C^\infty (]- \infty, 0[) \cap C^0(]- \infty, 0]), 
    \quad \eta_n \to \eta \; \text{strongly in $L^1 (\R)$}. 
\eq 
Next, we set $\rho_1(t, x): = \rho_{0n}(x)$, $u_1(t, x) : = u_0( x)$ and, given 
\be \label{e:passoenne}
\rho_{n}\in C^1(]0, T[ \times \R) \; u_n \in C^2 (]0, T[ \times \R) \;  \text{for every $T>0$} \eq
such that 
\[
\partial_x u_n = \rho_n z_n, \quad \| z_n \|_{L^\infty} \leq \| z_0\|_{L^\infty},
\]
we iteratively define $\rho_{n+1}$ by solving the Cauchy problem for nonlocal continuity equation
\be \label{e:cauenne}
     \left\{
      \begin{array}{ll}
                  \partial_t \rho_{n+1} + \partial_x [V(\xi_{n+1}, u_n) \rho_{n+1}] =0 \\
                  \rho_{n+1} (0, \cdot) = \rho_{0n+1}  \\
      \end{array}
     \right.
\eq
where $\xi_{n+1}$ is defined as in formula~\eqref{e:nlGARZ} with $\rho$ replaced by $\rho_{n+1}$ and $\eta$ by $\eta_{n+1}$. Existence and uniqueness results for~\eqref{e:cauenne} in the class of bounded functions are by now well known and established for instance in~\cite{KeimerPflug}. Note in particular that the analysis in~\cite{KeimerPflug} combined with the regularity of $u_n$ implies that $V(\xi_{n+1}, u_n)$ is a continuous function, which is Lipschitz continuous with respect to the space variable, so that the classical Cauchy Lipschitz Picard Lindelh\"of Theorem applies to the following Cauchy problem 
\be \label{e:icsenne}
     \left\{
      \begin{array}{ll}
                  \displaystyle{\frac{d X_{n+1}}{dt}} = V(\xi_{n+1}, u_n)(t, X_{n+1})  \\
                  \phantom{cc}\\
                   X_{n+1} (0, x) =x, \\
      \end{array}
     \right.
\eq
which defines the characteristic curve $X_{n+1}(t, \cdot)$. The solution $\rho_{n+1}$ can then be described by relying on the classical method of characteristics as $\rho_{n+1}(t, \cdot) \mathcal L^1 = X_{n+1}(t, \cdot)_\sharp \left(\rho_{n+1}(0, \cdot) \mathcal L^1\right)$, where the push-forward $\sharp$ is defined by~\eqref{e:pushfdef}. Since $\rho_{n+1}(0, \cdot) \ge 0$, this implies that 
\be \label{e:pos}
    \rho_{n+1} \ge 0
\eq
and hence that $\xi_{n+1} \ge 0$ owing to the definition of $\xi_{n+1}$ and to the inequality $\eta_{n+1} \ge 0$. Owing to~\eqref{e:V}, $0 \leq V(\xi_{n+1}, u_n) \leq V(0, u_n)$ and by recalling that $u_n$ is uniformly bounded owing to~\eqref{e:passoenne} we obtain a uniform $L^\infty$ bound on $V(\xi_{n+1}, u_{n})$. In particular, since $\rho_0$ is compactly supported, then $\rho_{n+1}(t,\cdot)$ is compactly supported for every $t>0$. By using~\cite[Corollary 5.3, Remark 5.4]{KeimerPflug} we also obtain $\rho_{n+1}\in C^1(]0, T[ \times \R)$ for every $T>0$. \\
{\bf Step 2:} we establish the bound 
\be \label{e:xienne}
      0 \leq \xi_{n+1}(t, x) \leq 1. 
\eq
We have already established the positivity of $\xi_{n+1}$ at the previous step. To establish the second inequality in~\eqref{e:xienne}, we compute the material derivative of $\xi_{n+1}$ by direct computations, and arrive at 
\be \label{e:xit}
     \partial_t \xi_{n+1} + V(\xi_{n+1}, u_n) \partial_x \xi_{n+1} = \int_x^{+\infty} \eta'_{n+1}(x-y) 
     \left[ V(\xi_{n+1}, u_n) (t, x) - V(\xi_{n+1}, u_n) (t, y) \right] \rho_{n+1} (t, y) dy, 
\eq 
which in particular implies that $\xi_{n+1}$ is Lipschitz continuous with respect to the $t$ variable as well.
Next, we recall that, for every $t\ge 0$, the function $\rho_{n+1}(t, \cdot)$ is compactly supported, which owing to Lebesgues's Dominated Convergence Theorem implies that $\xi_{n+1}(t, \cdot)$ vanishes at both $\pm \infty$. We conclude that, for every $t \ge 0$, the function  $\xi_{n+1}(t, \cdot)$ has a  maximum, and we set
\be \label{e:gi}
    g(t) : = \max_{x \in \R }  \xi_{n+1}(t, x),
\eq
which is a Lipschitz continous function being the supremum of Lipschitz continuous functions. Since $0 \leq \xi_{n+1}(0, \cdot) \leq 1$ owing to~\eqref{e:id}, to establish~\eqref{e:xienne} it suffices to show that $g' (t_\ast) \leq 0$ for a.e $t_\ast \in \R_-$ such that $g(t_\ast) \ge 1$. 

Towards this end, we recall the following elementary property: \emph{assume that $I \subseteq \R$ is an interval, and $f, g: I \to \R$ are two locally Lipschitz continuous functions, $f \leq g$. If both $f$ and $g$ are differentiable at $\tau \in I$ and $f(\tau) = g(\tau)$ then 
$f'(\tau) = g'(\tau)$.} To establish this property it suffices to set $h := g-f \ge 0$ and point out that if $h(\tau) =0$ then $\tau$ is a point of minimum and hence the differentiability of $h$ at $\tau$ yields $h'(\tau)=0$. 

Let us go back to the function $g$ defined by  
\eqref{e:gi}, fix a point of differentiability $t_\ast$ such that $g(t_\ast)\ge1$ and $x_\ast$ such that $\xi_{n+1}(t_\ast, x_\ast)= g(t_\ast)$. We denote by $a^\ast \in \R$ the point satisfying $X_{n+1}(t_\ast, a_\ast) = x_\ast$ and we set 
$$
    f(t) = \xi_{n+1} (t, X_{n+1}(t, a_\ast)).
$$
Owing to the elementary property we recalled above, to show that $g'(t_\ast) \leq 0$ it suffices to show that $f'(t_\ast) \leq 0$. This follows by combining~\eqref{e:xit} and~\eqref{e:V} with $\eta' \ge 0$. Note in particular that~\eqref{e:V} dictates that $V \ge 0$ and that $V(\xi, w) =0$ for every $\xi \ge 1$. \\
{\bf Step 3:} we now establish the bound 
\be \label{e:bdenne}
      \rho_{n+1} (t, x) \leq   \frac{1}{1- t C(V) \| z_0 \|_{L^\infty}} \quad \text{for every $(t, x) \in ]0, T[ \times \R$}, \; T = C(V) \| z_0 \|_{L^\infty}. 
\eq
Towards this end, we argue by induction. The above estimate holds for $n=1$. We assume that it is verified by $\rho_n$ and establish it for $\rho_{n+1}$. 
We recall that $\rho_{n+1}$ is continuously differentiable and compute the material derivative of $\rho_{n+1}$ as
\begin{equation*} 
\begin{split}
      \partial_t \rho_{n+1}   +  V(\xi_{n+1}, u_n) \partial_x \rho_{n+1} & 
=
      - \rho_{n+1} \partial_x [V(\xi_{n+1}, u_n) ]= 
     -  \rho_{n+1} \partial_1 V \partial_x \xi_{n+1} - \rho_{n+1}\partial_2 V \partial_x u_n \\
    & \stackrel{\eqref{e:passoenne}}{=}
     -  \rho_{n+1} \partial_1 V \partial_x \xi_{n+1} - \rho_{n+1}\partial_2 V \rho_n z_n 
\end{split}
\end{equation*} 
Next, we point out that, if $x_\ast$ is a point where the maximum of $\rho_{n+1}(t_\ast, \cdot)$ is attained then 
$$
    \partial_x \xi_{n+1} 
    (t_\ast, x_\ast) = - \rho_{n+1} (t_\ast, x_\ast) \eta_{n+1}(0) + \int_{x_\ast}^{+\infty} 
     \eta_{n+1}'(x_\ast-y)   \rho_{n+1}(t_\ast, y) dy \le 0 
$$
and by recalling that $\partial_1 V \leq 0$ and $\rho_{n+1} \ge 0$ this implies that the material derivative evaluated at 
$(t_\ast, x_\ast)$ satisfies 
\begin{equation} \label{e:inplace}
\begin{split}
\Big[\partial_t \rho_{n+1}   +  V(\xi_{n+1}, u_n)&  \partial_x \rho_{n+1}\Big]
     (t_\ast, x_\ast)  \leq - \rho_{n+1} 
    \partial_2 V \rho_n z_n (t_\ast, x_\ast)
    \leq C(V) \rho_{n+1} 
   | \rho_n z_n| (t_\ast, x_\ast) \\ & 
    \stackrel{\eqref{e:passoenne}}{\leq}
     C(V) \| z_0 \|_{L^\infty} \rho_{n+1} 
    \rho_n (t_\ast, x_\ast)
    \stackrel{\text{induction}}{\leq}  \frac{C(V) \| z_0 \|_{L^\infty}}{1- t C(V) \| z_0 \|_{L^\infty}}  
    \rho_{n+1} (t_\ast, x_\ast),
\end{split}
\end{equation}
where in the last inequality we have used the induction assumpion that~\eqref{e:bdenne} is satisfied by $\rho_n$.
With~\eqref{e:inplace} in place we can  basically argue as in {Step 2}: 
we recall that $\rho_{n+1}(t, \cdot)$ is a continuous and compactly supported function
for every $t$, set
\be \label{e:gienne}
    g_{n+1}(t) : = \max_{x \in \R} \rho_{n+1} (t, \cdot),
\eq
fix a point $t_\ast$ at which $g_{n+1}$ is differentiable and $x_\ast \in \R$ such that $g(t_\ast) = \rho_{n+1} (t_\ast, x_\ast).$ Next, we set 
$$
    f(t) : = \rho_{n+1} (t, X_{n+1} (t, a_\ast)),
$$
where $a_\ast$ satisfies $X_{n+1} (t_\ast, a_\ast)= x_\ast$. By combining the elementary property recalled at the previous step with~\eqref{e:inplace}
we conclude that 
$$
    g_{n+1}'(t_\ast) \leq \frac{C(V) \| z_0 \|_{L^\infty}}{1- t C(V) \| z_0 \|_{L^\infty}}  
    g_{n+1}(t_\ast), 
$$
and by the arbitrariness of $t_\ast$ and by using the Comparison Theorem for ODEs we arrive at~\eqref{e:bdenne}, which concludes our induction argument.   \\
{\bf Step 4:} we now define  $z_{n+1}$ and $u_{n+1}$. First, we recall~\eqref{e:icsenne}, set $Y_{n+1}(t, \cdot) : = X_{n+1}(t, \cdot)^{-1}$ and 
$
    z_{n+1} (t, x) : = z_{0n+1} (Y_{n+1}(t, x)),
$
which immediately yieds the inequality
\be \label{e:glicine}
    \| z_{n+1} \|_{L^\infty} \leq  \| z_{0} \|_{L^\infty}. 
\eq
Being the vector field $V(\xi_{n+1},u_n)$  continuously differentiable, so is the inverse flow $Y_{n+1}(t, \cdot)$ and hence $z_{n+1}$.
Note furthermore that 
\be \label{e:cammello}
   \partial_t [\rho_{n+1} z_{n+1}] + \partial_x [V(\xi_{n+1}, u_n ) \rho_{n+1} z_{n+1} ]=0
\eq
in the sense of distribution. Albeit well known, we provide a proof of~\eqref{e:cammello} for the sake of completeness. It suffices to test against test functions in the form $\phi(t) \nu (x)$, with $\phi \in C^\infty_c (]0, + \infty[)$ and $\nu \in C^\infty_c (\R)$. Owing to the identity  $\rho_{n+1}(t, \cdot) \mathcal L^1 = X_{n+1}(t, \cdot)_\sharp \left(\rho_{n+1}(0, \cdot) \mathcal L^1\right)$ we  have 
$$
   \int_{\R} \rho_{n+1}z_{n+1} (t, x)  \nu (x) dx  = 
    \int_{\R} \rho_{n+1} (t, x) z_0 (Y(t, x)) \nu (x)  dx
   =
    \int_{\R} \rho_{0 n+1}(y)  z_0 (y) \nu (X_{n+1}(t, y)) dy,
$$
which in turn implies owing to an integration by parts in the $t$ variable and recalling~\eqref{e:icsenne}
\[
\begin{split}
   \int_0^\infty & \partial_t \phi (t)  \int_{\R} \rho_{n+1}z_{n+1} (t, x)  \nu (x) dx dt =
    \int_0^\infty \partial_t \phi (t) \int_{\R} \rho_{0 n+1}(y)  z_0 (y) \nu (X_{n+1}(t, y)) dy dt 
    \\ & = - \int_0^\infty \phi(t) \int_{\R} \rho_{0 n+1}(y)  z_0 (y)   V(\xi_{n+1}, u_n) \partial_x \nu (X_{n+1}(t, y)) dy dt 
    \\ & = -
    \int_0^\infty \phi (t)  \int_{\R} \rho_{n+1}z_{n+1}  V(\xi_{n+1}, u_n) (t, x)    \partial_x \nu (x) dx dt,
\end{split}
\]
and by the arbitrariness of $\phi$ and $\nu$ this establishes~\eqref{e:cammello}. 
We now set
\be \label{e:uenne+1}
    u_{n+1} (t, x) = u_\infty + \int_{- \infty}^x \rho_{n+1} z_{n+1} (t, y) dy, 
\eq
where $u_\infty$ is the same as in~\eqref{e:uinfty}. Owing to~\eqref{e:cammello}, the continuously differentiable function $u_{n+1}$ satisfies
\be \label{e:te}
     \partial_t u_{n+1} + V(\xi_{n+1}, u_n) \partial_x u_{n+1} =0
\eq
on $\R_+ \times \R$. Owing to the classical method of characteristics, we get the bound
\be \label{e:uenne}
   \inf u_{0n+1} \leq u_{n+1}(t, x) \leq \sup u_{0n+1} \quad \text{for every $(t, x) \in \R_+ \times \R$},
\eq
where 
$
    u_{0 n+1} (x) 
$
is the same as in~\eqref{e:koala}. Moreover, since $\rho_{n+1}, z_{n+1}$ are continuously differentiable functions, then $u_{n+1}$ is of class $C^2$.  \\
{\bf Step 5:}  to conclude the definition of the approximate sequences, we are left to discuss the case where we have~\eqref{e:idgarz}.
 First, we fix a sequence of continuous functions $\psi_{0n}$ with $\|\psi_{0n}\|_{L^\infty}\le \|\psi_{0}\|_{L^\infty}$ and converging pointwise a.e. to $\psi_0$. Next, we point out that~\eqref{e:idgarz} implies that the function $z_0' \in L^1(\R)$, which in turn implies that
\be \label{e:cervo}
   z_\infty = \lim_{x \to - \infty} z_0(x)
\eq
exists and is finite. Next, we set $\psi_{n+1}(t, x): = \psi_0 (Y_{n+1}(t, x))$, which implies that $\psi_{n+1}$ is a continuous function, and 
\be \label{e:prugna}
     z_{n+1} (t, x) : = z_\infty + \int_{- \infty}^x \rho_{n+1} \psi_{n+1} (t, y) dy.
\eq
Notice that the functions $z_{n+1}$ are Lipschitz continuous in the space variable $x$ uniformly with respect to $n$ for $t<T$, provided $T$ is the same as in~\eqref{e:bdenne}.
Note that in this case $z_{n+1}$ is not the same as in the previous cases, but this is irrelevant since they converge to the same limit: in particular the initial datum $z_{n+1}(0, \cdot)$ converges to $z_0$ in $C^0 (\R)$ owing to~\eqref{e:rho0enne} and~\eqref{e:prugna}. Note furthermore that by arguing as at  Step 4 we infer that the continuously differentiable function $z_{n+1}$ satisfies $\partial_t z_{n+1} + V(\xi_{n+1}, u_n) \partial_x z_{n+1}=0$ and by using the classical method of characteristics this in turn implies that  $\| z_{n+1} \|_{L^\infty} \leq 
\| z_{n+1} (0, \cdot) \|_{L^\infty}$, which converges to $\| z_0 \|_{L^\infty}$ as $n$ goes to $+ \infty$. 
We then define $u_{n+1}$ as before by using~\eqref{e:uenne+1}. \\
{\bf Step 6:} we show that, if $\rho_0 \in W^{1, \infty}(\R)$, then  
\be \label{e:estder}
   \| \partial_x \rho_{n+1} (t, \cdot) \|_{L^\infty},  \| \partial_{xx} u_{n} (t, \cdot) \|_{L^\infty}  \leq 
   C(t, \| \rho'_0\|_{L^\infty}, t, \eta, V, \bar t, \| \psi_0 \|_{L^\infty},  z_\infty, \| \rho_0 \|_{L^1}), \quad 
   \text{for every $t \in ]0, T[$},
\eq
where $T$ is the same as in~\eqref{e:bdenne}. First, we recall the representation formula 
{\color{black}
$$
   \rho_{n+1}(t, X_{n+1}(t, x))= \rho_{0n+1}(x) \exp \left[ \int_0^t \partial_x  V(\xi_{n+1}, u_n )(s, X_{n+1}(s, x))ds  \right], 
$$
}
which yields 
{\color{black}
\be \label{e:derivatona}
\begin{split}
   \partial_x &\rho_{n+1}(t, X_{n+1}(t, x)) \frac{\partial X_{n+1}} {\partial x} (t,x)  = \rho_{0n+1}'(x)  \exp \left[ \int_0^t \partial_x  V(\xi_{n+1}, u_n )(s, X_{n+1}(s, x))ds  \right] \\ 
   &  + 
   \rho_{0n+1}(x)  \exp \left[ \int_0^t \partial_x  V(\xi_{n+1}, u_n )(s, X_{n+1}(s, x))ds  \right] \int_0^t \partial_{xx}  V(\xi_{n+1}, u_n )(s, X_{n+1}(s, x)) \frac{\partial X_{n+1}} {\partial x}(s,x)  ds.
\end{split}
\eq
}
Since 
{\color{black}
\be \label{e:derflusso}
    \frac{\partial X_{n+1}} {\partial x} (t, x) = \exp \left[ \int_0^t \partial_x  V(\xi_{n+1}, u_n )(s, X(s, x))ds  \right]
\eq}
from~\eqref{e:derivatona} we get 
{\color{black}
\be \label{e:derivatona2}
   \partial_x \rho_{n+1}  (t, X_{n+1}(t, x)) = \rho_{0n+1}' (x)  +   \rho_{0n+1} (x) \int_0^t \partial_{xx}  V(\xi_{n+1}, u_n ) (s, X_{n+1}(s, x)) \frac{\partial X_{n+1}} {\partial x}(s,x) ds.
\eq}
We now recall that $\rho_{n+1}(t, \cdot)$ is a continous and compactly supported function, set 
$$
    \ell_{n+1}(t): = \max_{x \in \R} | \partial_x \rho_{n+1} (t, x)|
$$
and use~\eqref{e:derivatona2} to control it. First of all we point out that, if $\rho_0 \in W^{1 \infty} (\R)$, we can construct the approximating sequence in~\eqref{e:rho0enne} in such a way that $\| \rho'_{0n} \|_{L^\infty} \leq \| \rho_0'\|_{L^\infty}$. Next, we point out that 
\be \label{e:natale2}
|\partial_x V(\xi_{n+1},u_n)| =   |\partial_1 V(\xi_{n+1},u_n) \partial_x \xi_{n+1} + \partial_2 V(\xi_{n+1},u_n) \partial_x u_n| \eq
and that
\be \label{e:natale}
\begin{split}
|\partial_{xx} V(\xi_{n+1},u_n)|=  &~ \big |\partial_{11} V(\xi_{n+1},u_n) (\partial_x \xi_{n+1})^2 + 2 \partial_{12} V(\xi_{n+1},u_n) \partial_x \xi_{n+1}\partial_x u_n \\
&~ + \partial_1 V(\xi_{n+1},u_n) \partial_{xx}\xi_{n+1} + \partial_{22} V(\xi_{n+1},u_n) (\partial_xu_n)^2 + \partial_2 V(\xi_{n+1},u_n) \partial_{xx}u_n|.
\end{split}
\eq
Note that 
\be \label{e:pasqua}\begin{split}
   |\partial_x \xi_{n+1}(t, x) | & = \left| - \rho_{n+1} (t, x) \eta_{n+1}(0) + \int_{x}^{+\infty} 
     \eta_{n+1}'(x-y)   \rho_{n+1}(t, y) dy \right| \leq C(\eta) \| \rho_{n+1} (t, \cdot) \|_{L^\infty}  \\ &
     \stackrel{\eqref{e:bdenne}}{\leq} 
     C(\eta, V, \bar t, \| z_{0n} \|_{L^\infty}) 
     \end{split}
     \eq
provided $t \leq \bar t < T$ and $T$ is the same as in~\eqref{e:bdenne}.  Similarly, $\| \partial_x u_n \|_{L^\infty} = \| \rho_n z_n \|_{L^\infty} \leq  C(\eta, \bar t, V, \| z_{0n} \|_{L^\infty})$. Also, 
\be \label{e:pasquetta}
\begin{split}
|\partial_{xx}\xi_{n+1}|(t,x) = &~  \left| \partial_x \left(-\eta_{n+1}(0)\rho_{n+1}(t,x) + \int_{-\infty}^0\eta_{n+1}'(y)\rho_{n+1}(t,x-y) dy  \right)   \right| \\
\le &~ 2 \|\partial_x \rho_{n+1}(t)\|_{L^\infty}\|\eta_{n+1}'\|_{L^1(\R^-)} 
\le~ C(\eta) \ell_{n+1}(t)
\end{split}
\eq
and 
\be \label{e:1maggio}
   |\partial_{xx} u_n (t, x)| \leq | \partial_x \rho_n z_n + \rho^2_{n} \psi_n| \stackrel{\eqref{e:bdenne}}{\leq} \|z_{0n} \|_{L^\infty} \ell_n + C(V, \|z_{0n}\|_{L^\infty},V,  \| \psi_0 \|_{L^\infty}).
\eq
We now plug~\eqref{e:pasqua},\eqref{e:pasquetta} and~\eqref{e:1maggio} into~\eqref{e:natale} and use it to control~\eqref{e:derivatona}, noting  that owing to~\eqref{e:derflusso},\eqref{e:natale2} and~\eqref{e:pasqua} we have $|\sfrac{\partial X_{n+1}}{\partial x}| \leq C(\eta, V, \bar t, \| z_{0n} \|_{L^\infty})$. We eventually arrive at 
\be \label{e:carnevale}
\begin{split}
   \ell_{n+1} (t) & \leq \| \rho_0' \|_{L^\infty} + C(\eta, V, \bar t, \| z_{0n }\|_{L^\infty}) \int_0^t [ \ell_{n}(s) + \ell_{n+1} (s) + 1 ] ds 
   \\ & \stackrel{\eqref{e:prugna}}{\leq} 
    \| \rho_0' \|_{L^\infty} + \underbrace{C(\eta, V, \bar t, \| \psi_0 \|_{L^\infty},  z_\infty, \| \rho_0 \|_{L^1})}_{:= \kappa}\int_0^t [ \ell_{n}(s) + \ell_{n+1} (s) + 1 ] ds, \quad \text{for every $t \leq \bar t$}
 \end{split}
\eq
We now argue by induction to show that~\eqref{e:carnevale} implies 
\be \label{e:25aprile}
   \ell_{n} (t) \leq \frac12 [1 + 2\| \rho_0'\|_{L^\infty}]\exp [2\kappa t] -\frac12 \quad \text{for every $n$}, 
\eq
which yields the control on $\partial_x \rho_{n+1}$ and, by using~\eqref{e:1maggio}, on $\partial_{xx} u_n$ in~\eqref{e:estder}.
The inequality in~\eqref{e:25aprile} is satisfied for $n=0$ because $\ell_0(t)= \| \rho_{0n}' \|_{L^\infty}.$ To establish the induction step, we term $p$ the solution of the ODE $p'= \kappa [1 + {\color{black}2}p]$ attaining value $ \| \rho_0'\|_{L^\infty}$ at $t=0$. Note that 
$$
   p(t) =  \| \rho_0' \|_{L^\infty} +\kappa \int_0^t [ {\color{black}2} p(s) + 1 ] ds
$$
and that~\eqref{e:25aprile} (that is, the induction assumption) implies that $\ell_n(t) \leq p(t)$. By using~\eqref{e:carnevale} we arrive at 
$$
    \ell_{n+1}(t) - p(t) \leq \kappa \int_0^t [ \ell_{n+1}- p  ](s) ds,
$$
which yields $\ell_{n+1(t)} \leq p(t)$ and hence concludes the induction step.
\subsection{Passage to the limit}\label{sss:limite} 
We fix $\bar t < T$, where $T$ is the same as in~\eqref{e:bdenne},  and from~\eqref{e:pos} and~\eqref{e:bdenne} deduce that there is a subsequence $\{ \rho_{n_k} \}_{k \in \mathbb N}$ such that 
\be \label{e:convrho}
     \rho_{n_k} \weaks \rho \quad \text{weakly$^\ast$ in $L^\infty(]0, \bar t[ \times \R)$}
\eq
for some limit bounded function $\rho$. Since $\eta_{n_k} \to \eta$ in $L^1 (\R)$ owing to~\eqref{e:etaenne}, 
by combining~\eqref{e:xienne} with Lebesgue's Dominated Convergence Theorem we conclude that 
\be \label{e:convxi} 
       \xi_{n_k} \to \xi \quad \text{strongly in $L^1_{\mathrm{loc}}([0, \bar t[ \times \R)$},
\eq
provided $\xi$ is the same as in~\eqref{e:nlGARZ}. 
Next, we recall~\eqref{e:uenne} and the fact that $u_{0n}$ converges uniformly to $u_0$. We also combine the decomposition $\partial_x u_n = \rho_n z_n$ with the bounds~\eqref{e:bdenne} and $\| z_n \|_{L^\infty} \leq \| z_0 \|_{L^\infty}$, which yields a uniform bound on $\partial_x u_n$. Next, we
use~\eqref{e:te} to deduce a uniform bound on $\partial_t u_n$, and this finally yields the equicontinuity 
of the sequence $\{ u_n \}$. We apply the Arzel\`a Ascoli Theorem 
and conclude that there is a subsequence of $\{ u_{n_k} \}$ (which to simplify notation we do not relabel) such that 
\be \label{e:convu}
      u_{n_k} \to u \; \text{in $C^0(\Omega)$ for every $\Omega \subseteq [0, \tilde T] \times \R$ compact}
\eq
and 
$$
     \partial_t u_{n_k} \weaks \partial_t u, \quad \partial_x u_{n_k} \weaks \partial_x u 
      \; \text{weakly$^\ast$ in $L^\infty(]0, \tilde T[ \times \R)$},
$$
for some Lipschitz continuous function $u$. Assume for a moment that we have shown that 
\be \label{e:convu2}
      u_{n_k-1} \to u \; \text{in $C^0(\Omega)$ for every $\Omega \subseteq [0, \tilde T] \times \R$ compact},
\eq
where $u$ is the same as in~\eqref{e:convu}\footnote{Note indeed that by compactness $\{ u_{n_k-1}  \}_{k \in \mathbb N}$ has a converging subsequence, but its limit might in principle differ from $u$.}, which will be done in the next paragraph. 
Then by combining~\eqref{e:convu2} and~\eqref{e:convxi} 
we have $V(\xi_{n_k}, u_{n_k-1}) \to V(\xi, u)$ strongly in $L^1_{\mathrm{loc}}(]0, \tilde T[ \times \R)$. This implies that 
we can pass to the limit in the distributional formulation of~\eqref{e:cauenne} and in the 
pointwise formulation of~\eqref{e:te}, and obtain a solution of~\eqref{e:nlGARZ} satisfying~\eqref{e:localex},~\eqref{e:mpxi} and also the first inequality in~\eqref{e:mpii}.

To conclude the proof, we are left to establish the equality $\partial_x u = \rho z$ for some $z$ satisfying~\eqref{e:mpii}. Towards this end, we set 
\be \label{e:orso}
z(t, x) : = z_0 (Y(t, x)),
\eq
where $Y(y, \cdot) = X(t, \cdot)^{-1}$ and $X$ is defined as in~\eqref{e:icsenne} with $\xi_{n+1}$ and $u_n$ replaced by $\xi$ and $u$, respectively. Note that $V(\xi, u)$ is a regular vector field and hence the classical Cauchy Lipschitz Picard Lindelh\"of Theorem still applies. Note furthermore that $\| z \|_{L^\infty} \leq \| z_0 \|_{L^\infty}$. 
We now set 
$$
    w(t, x) =  u_{\infty} + \int_{- \infty}^x \rho z(t, y) dy,
 $$
where $u_\infty$ is the same as in~\eqref{e:uinfty}, and point out that $w(0, \cdot) = u_0$.  By arguing as in {Step 4} in \S\ref{sss:app} we obtain $\partial_t w + V(\xi, u) \partial_x w=0$. Since this transport equation has a unique solution satisfying $w(0, \cdot) = u_0$, by recalling the equation at the first line of~\eqref{e:GARZ} we conclude that $w=u$ and this establishes the identity $\partial_x u = \rho z$. 

We are now left to establish the regularity results. Towards this end, we assume~\eqref{e:idgarz} and set $\psi (t, x) : = \psi_0 (Y(t, x))$ and 
$$
    z(t, x) =  z_{\infty} + \int_{- \infty}^x \rho \psi (t, y) dy,
$$
where $z_{\infty}$ is the same as in~\eqref{e:prugna}, and by arguing as before this yields $\partial_x z = \rho \psi$, that is~\eqref{e:est_h}.  To establish~\eqref{e:est_rho}, we pass to the limit in~\eqref{e:estder} to control the space derivatives, then use the first equation in~\eqref{e:nlGARZ} to control $\partial_t \rho$ and the derivative of  second equation in~\eqref{e:nlGARZ} to control $\partial_{tx} u$.
This concludes the proof of the existence part of Theorem~\ref{t:wpnl}. 
\subsection{Proof of~\eqref{e:convu2}} \label{sss:pconvu2}
We use an argument inspired by~\cite{ColomboCrippaSpinolo2} and introduce the functional 
\be \label{e:quenne}
   Q_n(t) : = \int_{\R} \rho_0( x) |X_{n+1}(t, x) - X_n(t, x)| dx,
\eq
where $X_{n+1}$ and $X_n$ are defined as in~\eqref{e:icsenne}. We point out that~\eqref{e:eta} implies that $\eta$ is a bounded variation function and state the following result. 
\begin{lemma}\label{l:quenne}
There is $\delta>0$, only depending on depending on the Lipschitz constant of $V$,  $\| z_0 \|_{L^\infty}$ and $\mathrm{Tot Var} \ \eta$ and on the expression at the right had side of~\eqref{e:bdenne},
such that 
\be\label{e:vquenne}
   \sup_{t \in [0, \delta] }Q_n (t ) \to 0 \quad \text{as $n \to +\infty$}.
\eq 
\end{lemma}
We postpone the proof of Lemma~\ref{l:quenne} to the next paragraph and we show that~\eqref{e:vquenne} implies~\eqref{e:convu2}. Up to extracting a further subsequence (which to simplify notation we do not relabel) we can assume that  
\be \label{e:convu3}
      u_{n_k-1} \to v \; \text{in $C^0(\Omega)$ for every $\Omega \subseteq [0, \tilde T] \times \R$ compact},
\eq
for some continuous limit function $v$. Establishing~\eqref{e:convu2} amounts to show that $u=v$. We fix $x \in \R$ and recall that $\{ X_n (\cdot, x)\}$ is defined as in~\eqref{e:icsenne} and also recall from {Step 1} in \S\ref{sss:app} that the Lipschitz constant on $\{ X_n (\cdot, x) \}$ is uniformly bounded. We apply the Arzel\`a Ascoli Theorem and conclude that we can extract a further subsequence (which we do relabel to simplify the notation, and which in principle depends on $x$) such that 
$$
   X_{n_k} (\cdot, x) \to X(\cdot, x), \quad X_{n_k-1} (\cdot, x) \to Z(\cdot, x)
$$
for some limit functions $X( \cdot, x)$ and $Z(\cdot, x)$ that solve 
\be \label{e:ics}
     \left\{
      \begin{array}{ll}
                  \displaystyle{\frac{d X}{dt}} = V(\xi, u)(t, X)  \\
                  \phantom{cc}\\
                   X (0, x) =x, \\
      \end{array}
     \right. \qquad 
     \left\{
      \begin{array}{ll}
                  \displaystyle{\frac{d Z}{dt}} = V(\xi, v)(t, Z)  \\
                  \phantom{cc}\\
                   Z (0, x) =x, \\
      \end{array}
     \right. 
\eq
respectively. By the Lipschitz continuity of $V(\xi, u)$ and $V(\xi, v)$, each of the above Cauchy problems has a unique solution. This implies that the \emph{whole} sequences $\{ X_{n_k} (\cdot, x)\}$ and $\{ X_{n_k-1} (\cdot, x)\}$ converge, with no need to pass to subsequences. In particular, the convergence holds \emph{for every $x$}. By using the Dominated Convergence Theorem we then conclude that 
\be \label{e:vquenne2}
   \int_\R \rho_0( x) |X(t, x) - Z(t, x)| dx = \lim_{k \to + \infty} Q_{n_k}(t) \stackrel{\eqref{e:vquenne}}{=}0, \quad \text{for every $t \in [0, \delta]$}.
\eq
We now fix $t \in [0, \delta]$ and by passing to the limit in the identities $u_{n_k}(t, X_{n_k}(t, x))=u_0(x) = u_{n_k-1}(t, X_{n_k-1}(t, x))$ we arrive at
$
    u(t, X(t, x))= u_0(x) = v (t, Z(t, x))),
$ which can be rewritten as  
\be \label{e:uuu2}
   u (t, z) = u_0 (Y_1 (t, z)), \quad v (t, z) = u_0 (Y_2 (t, z)).
\eq
provided $Y_1(t, \cdot)$ and $Y_2(t, \cdot)$ are  the inverse functions of $X(t, \cdot)$ and $Z (t, \cdot)$, respectively. Let us now fix $z \in \R$: if $Y_1 (t, z) = Y_2 (t, z)$ then $u(t, z) = v (t, z)$, so in the following we focus on the case $Y_1 (t, z) \neq Y_2 (t, z)$ and just to fix the ideas we assume $Y_1 (t, z) < Y_2 (t, z)$. 
Assume for a moment that we have proved that 
\be \label{e:rho002}
    \rho_0 (x) =0 \quad \text{a.e. on $[Y_1 (t, z), Y_2 (t, z)]$}, 
\eq
then owing to~\eqref{e:id2} we have $u'_0 \equiv 0 $ a.e. on $[Y_1 (t, z), Y_2 (t, z)]$ and hence $u_0 (Y_1 (t, z)) = u_0 (Y_2 (t, z))$, which owing to~\eqref{e:uuu2} implies $u(t, z) = v (t, z)$. To establish~\eqref{e:rho002} we argue by contradiction and assume that the set of points $x \in ]Y_1 (t, z), Y_2 (t, z)[$ such that $\rho_0(x) >0$ has positive measure. Owing to~\eqref{e:vquenne2}, this implies that there is  $x_0 \in ]Y_1 (t, z), Y_2 (t, z)[$ with $X (t, x_0) = Z (t, x_0)$. Since the maps $x \mapsto X(t, x)$ and $x \mapsto Y(t, x)$ are both strictly monotone, we have 
$$
    z = X(t, Y_1 (t, z)) < X (t, x_0), \quad  Z (t, x_0) <  Z(t, Y_2 (t, z)) = z, 
$$
which contradicts the equality $X (t, x_0) = Z (t, x_0)$ and hence establishes~\eqref{e:rho002}.  This in turn establishes the identity $u =v$ and hence the local in time existence on the time interval $[0, \delta]$. To establish the existence on the whole time interval $]0, \bar t[$ we point out that the right-hand side of~\eqref{e:bdenne} is uniformly bounded on $]0, \bar t[$ and conclude that we can iterate the above argument. 
\subsection{Proof of Lemma~\ref{l:quenne}}
\label{ss:proofl21}
We have 
\begin{equation*}
\begin{split}
     \frac{d Q_n}{dt} &\stackrel{\eqref{e:icsenne}}{\leq} \int_{\R} |V (\xi_{n+1}, u_n)(t, X_{n+1}) - V(\xi_n, u_{n-1})(t, X_n) | \rho_0(x) dx \\
     & \leq  \underbrace{\int_{\R} |V (\xi_{n+1}, u_n)(t, X_{n+1}) - V (\xi_{n+1}, u_n)(t, X_{n}) | \rho_0(x)dx}_{:=I_1} \\ &  +
      \underbrace{\int_{\R} |V(\xi_{n+1}, u_n)(t, X_{n}) - V(\xi_n, u_{n-1})(t, X_n) | \rho_0(x) dx }_{:=I_2}
\end{split}
\end{equation*}
and 
$$
    I_1  \leq C(V) \underbrace{\int_{\R} |\xi_{n+1}(t, X_{n+1}) - \xi_{n+1}(t, X_{n}) | \rho_0(x) dx }_{:=I_{11}}+ 
   C(V)  \underbrace{\int_{\R} |u_n (t, X_{n+1}) - u_{n}(t, X_{n}) | \rho_0(x) dx}_{:=I_{12}}. 
$$
We have 
$$
  I_{12} \leq \| \partial_x u_n \|_{L^\infty} \int_{\R} |X_{n+1}(t, x) - X_n(t, x) | \rho_0(x)| dx \stackrel{\partial_x u_n = \rho_n z_n,\eqref{e:glicine}}{\leq} 
  \| \rho_n \|_{L^\infty} \| z_0 \|_{L^\infty} Q_n (t),
$$
and 
\begin{equation*}
\begin{split}
    I_{11}
   & =  \int_{\R}   \rho_0(x) \left| \int_{\R} [ \eta (X_{n+1}(t, x) - y ) - \eta (X_n(t, x) - y)] \rho_{n+1} (t, y)|  dy\right| dx\\
   &  
 \leq \| \rho_{n+1} \|_{L^\infty}  \mathrm{Tot Var}\ \eta \int_{\R}   \rho_0(x) |X_{n+1} (t, x)- X_n(t, x) |   dx 
     = \| \rho_{n+1} \|_{L^\infty} \mathrm{Tot Var} \ \eta \,Q_n(t) .
\end{split}
\end{equation*}
In the previous expression, we have used the characterization of bounded total variation functions through different quotients. To control $I_2$ we point out that 
$$
     I_2  \leq C(V) \underbrace{
    \int_{\R} |\xi_{n+1}(t, X_n) - \xi_n(t, X_n) | \rho_0(x) dx}_{:=I_{21}} + C(V) \underbrace{
    \int_{\R} |u_n(t, X_n) - u_{n-1}(t, X_n) | |\rho_0(x) dx}_{:=I_{22}}.
$$
To control $I_{21}$ we use the identity $\rho_n (t, \cdot) \mathcal L^1= X_n (t, \cdot)\# \left(\rho_0\mathcal L^1\right)$, namely
\be 
\label{e:pushf}
     \int_\R \rho_n(t, y) \varphi(y) dy =   \int_\R \rho_0(x) \varphi( X_n(t, x)) dx
     \quad \text{for every $\varphi \in C^0_b(\R)$ }
\eq
and point out that, since $\rho_n$ is bounded, by approximation the above identity holds for every $\varphi \in L^1(\R)$.
This yields \begin{equation*}
\begin{split}
    I_{21} & 
    = \int_\R \rho_0(x) \left|  \int_\R \eta (X_n(t, x)  - y) [ \rho_{n+1} - \rho_n](y) \right| dy  dx \\ &
     \stackrel{\eqref{e:pushf}}{=}
    \int_\R \rho_0(x) \left| \int_\R  \big[ \eta (X_n (t, x) - X_{n+1} (t, y))- \eta (X_n (t, x) - X_n (t, y)) \big] \rho_0(y) dy\right|     dx \\ 
    & \leq  \int_\R \rho_0(y) \int_\R  \left| \eta (X_n (t, x) - X_{n+1} (t, y))- \eta (X_n (t, x) - X_n (t, y)) \right| \rho_0(x) dx    dy 
    \\ &  \stackrel{\eqref{e:pushf}}{=} 
     \int_\R \rho_0(y) \int_\R  \left| \eta (x - X_{n+1} (t, y))- \eta (x - X_n (t, y)) \right| \rho_n(t, x) dx    dy \\
    & \leq  \| \rho_n \|_{L^\infty} \mathrm{Tot Var} \ \eta \,Q_n(t) .\phantom{\int}
\end{split}
\end{equation*}
To control $I_{22}$ we point out that 
\be \label{e:lungochar}
    u_n (t, X_n(t, x))= u_0(x) = u_{n+1} (t, X_{n+1}(t, x))
\eq 
and decompose $I_{22}$ as follows: 
\begin{equation*}
\begin{split}
    I_{22} = & 
     \int_{\R} |u_n(t, X_n) - u_{n-1}(t, X_n) | |\rho_0(x) dx \\ &  \leq \underbrace{\int_{\R} |u_n(t, X_n) - u_{n-1} (t, X_{n-1}) | |\rho_0(x) dx}_{=0 \; \text{by~\eqref{e:lungochar}}}+
     \int_{\R} |u_{n-1}(t, X_{n-1}) - u_{n-1}(t, X_n) | \rho_0(x) dx  \\
     & \leq \| \partial_x u_{n-1} \|_{L^\infty} \int_{\R} |X_n(t, x) - X_{n-1}(t, x) | \rho_0(x) dx \stackrel{\partial_x u_n = \rho_n z_n,\eqref{e:glicine}}{\leq}  \| \rho_{n-1} \|_{L^\infty} \| z_0 \|_{L^\infty}  Q_{n-1}(t).
\end{split}
\end{equation*}
We now control the $L^\infty$ norm of $\rho_n$ and $\rho_{n+1}$ through~\eqref{e:bdenne} and conclude that 
$$
    \frac{d Q_n}{dt} \leq K [Q_n(t) + Q_{n-1} (t)], 
$$
for a suitable constant $K$ only depending on the Lipschitz constant of $V$, $\| z_0 \|_{L^\infty}$, $\mathrm{Tot Var} \ \eta$ and on the expression at the right had side of~\eqref{e:bdenne} evaluated at $t=\tilde T$. 
By recalling that $Q_n(0) =0$ the above inequality implies 
$$
    \sup_{s \in [0, t]} Q_n (t) \leq [e^{K t}-1]  \sup_{s \in [0, t]} Q_{n-1} (t)
$$
and from there a classical induction argument yields~\eqref{e:vquenne} provided $|\exp(K \delta) -1|<1$. This concludes the Proof of Lemma~\ref{l:quenne}. 
\begin{remark}
\label{r:mp}
Assume that $z_0 \ge 0$ and recall the definition of $z_{n}$ given at the beginning of {Step 4} in \S\ref{sss:app}, which implies $z_{n} \ge0$. By plugging this inequality in~\eqref{e:inplace} and recalling that $\rho_n \ge 0$, $\rho_{n+1} \ge 0$, $\partial_2 V \ge 0$ we conclude that the material derivative evaluated at $(t_\ast, x_\ast)$ is non-positive, which in turns yields the maximum principle $0 \leq \rho \leq \essup \rho_0$. We can then iteratively apply Theorem~\ref{t:wpnl} and obtain a global-in-time existence result (we also obtain uniqueness, by Lemma~\ref{l:uni} in the following paragraph). 
\end{remark}
\subsection{Uniqueness}
\label{ss:uninl}
\begin{lemma}\label{l:uni}
Fix $\tau>0$ and assume that $(\rho_1, u_1)$ and $(\rho_2, u_2)$ are two solutions of~\eqref{e:nlGARZ},\eqref{e:id} belonging to the class $L^\infty(]0, \tau[ \times \R) \times L^\infty (]0, \tau[; W^{1, \infty} (\R))$. Then $\rho_1 = \rho_2$ and $u_1=u_2$. 
\end{lemma}
\begin{proof}
We basically follow the same argument as in~\cite{ColomboCrippaSpinolo2} and in the proof of Lemma~\ref{l:quenne}, so we only touch upon the main steps.  For $i=1, 2$, we consider $\xi_i$ given by the formula in~\eqref{e:nlGARZ} with $\rho_i$ in place of $\rho$, and the Cauchy problem
\be \label{e:char}
       d X_i/dt = V(\xi_i, u_i), \qquad  X_i(0, x) = x,
\eq
which satisfies the assumptions of the Lipschitz Cauchy Picard Lindel\"of Theorem and hence admits exactly one solution, the characteristic curve $X_i(\cdot, x)$. Next, we set 
\be \label{e:Q}
    Q(t) : = \int_{\R} |X_1 (t, x)- X_2(t, x)| \rho_0(x) dx, 
\eq 
and by repeating the same argument as in the proof of Lemma~\ref{l:quenne} we arrive at 
\begin{equation*}
\begin{split}
     \frac{d Q}{dt} & \leq C (\| \partial_x u_1 \|_{L^\infty},  \| \rho_1 \|_{L^\infty}, \| \rho_2 \|_{L^\infty}, \mathrm{TotVar} \, \eta ) Q(t),
\end{split}
\end{equation*}
which owing to the Gronwall Lemma and the equality $Q(0)$ yields $Q(t)=0$ for every $t \in ]0, \tau[$. 
By arguing as in the proof of Lemma~\ref{l:quenne} we then obtain the identity $u_1 = u_2$. 
To establish the identity $\rho_1 = \rho_2$ we point out that
for every $\varphi \in C^0_c(\R)$ we have
$$
   \int_{\R} \rho_1 (t, x) \varphi(x) dx =       \int_{\R} \rho_0 (y) \varphi(X_1(t, y)) dy \stackrel{Q(t) \equiv 0}{=}
    \int_{\R} \rho_0 (y) \varphi(X_2(t, y)) dy  =   \int_{\R} \rho_2 (t, x) \varphi(x) dx 
$$
and by the arbitrariness of $\varphi$ we conclude that $\rho_1= \rho_2$. 
\end{proof}
\subsection{Global-in-time existence}\label{ss:gte}
\begin{lemma} \label{l:gte}
Under the same assumptions as in Theorem~\ref{t:wpnl}, assume furthermore that  $\eta$ is Lipschitz continuous on $\R_-$ and that 
\be \label{e:negativo}
   \partial_1 V  \eta(0) + \| z_0 \|_{L^\infty}\partial_2 V  \leq 0.
\eq
Then the solution given in the statement of Theorem~\ref{t:wpnl} can be extended globally-in-time. \end{lemma}
For other conditions yielding global in time existence, see also Remark~\ref{r:mp} at the end of \S\ref{ss:proofl21}. Note furthermore that, if we replace $\eta$ by $\eta_\ee : = \eta (\sfrac{\cdot}{\ee})/\ee$, Lemma~\ref{e:negativo} implies that, if $\eta \in W^{1 \infty} (\R_-)$ and $\partial_2 V$ is bounded away from $0$, there is $\ee_0$ such that for every $\ee \in ]0, \ee_0[$ condition~\eqref{e:negativo} is satisfied and hence the solution of~\eqref{e:nlGARZee} is defined globally in time. 
\begin{proof}[Proof of Lemma~\ref{l:gte}]
The analysis in \S\ref{sss:app} implies that to establish the proof of Lemma~\ref{l:gte} it suffices to replace~\eqref{e:bdenne} with a global-in-time bound. We argue as in Step 3 in \S\ref{sss:app} and conclude that the function $g_{n+1}$ defined by~\eqref{e:gienne} satisfies 
{\color{black}$$
    g'_{n+1}(t) \leq \partial_1 V \eta(0) g_{n+1}^2 (t) +  \underbrace{ \|\partial_1 V\|_{L^\infty(0,1)} \| \eta' \|_{L^\infty(\R_-)} \| \rho_0 \|_{L^1}}_{:=\hat c} g_{n+1}(t) + g_{n+1}(t) g_{n}(t) \| z_0 \|_{L^\infty}\partial_2 V.
$$}
We now use induction to show that, under~\eqref{e:negativo}, the above inequality implies $g_{n}(t) \leq \exp[ \hat c t ]$. Indeed, assume that $g_n$ satisfies the desired bound, then at any point where $g_{n+1}(t) = \exp[\hat ct]$ we have 
$$
     g'_{n+1}(t) \leq g_{n+1}^2 (t) \underbrace{[ \partial_1 V \eta(0)  +  \| z_0 \|_{L^\infty}\partial_2 V ]}_{\leq 0 \; \text{by~\eqref{e:negativo}}}
+ \hat c g_{n+1}(t) \leq \hat c g_{n+1}(t),
$$
and since $g_{n+1}(0) \leq 1$ this implies that $g_{n+1}$ cannot exceed $\exp[\hat ct]$. 
\end{proof}
\subsection{Stability}\label{ss:stab}
\begin{corol}\label{c:stability}
Assume $\{ \rho_{0k} \}_{k \in \mathbb N}$ and $\{ u_{0k} \}_{k \in \mathbb N}$ are sequences of initial data satisfying the same assumptions as in the statement of Theorem~\ref{t:wpnl}, and term $\{ (\rho_k, u_k) \}$ the corresponding sequence of solutions of the Cauchy problem obtained by coupling~\eqref{e:nlGARZ} with the initial datum given by $(\rho_{0k}, u_{0k})$ , and by $\{ \xi_k \}_{k \in \mathbb N}$ the sequence of convolution terms defined as in~\eqref{e:nlGARZ}. Assume 
$$
   \rho_{0k} \weaks \rho_0 \; \text{weakly$^\ast$ in $L^\infty (\R)$}, \quad u_{0k} \to u_0 \; \text{uniformly in $C^0(\R)$}, \quad \| z_{0k} \|_{L^\infty}, \| \rho_{0k} \|_{L^1} \leq M
$$
for some $M>0$. Then 
\be \label{e:stabilitynl}
        \rho_{k} \weaks \rho \; \text{weakly$^\ast$ in $L^\infty (]0, T[\times \R)$}, \quad u_{k} \to u \; \text{uniformly in $C^0 (\Omega)$, for every $\Omega \subseteq [0, T[ \times \R$ compact}
        \eq
 and $\xi_k (t, x) \to \xi (t, x)$ for every $(t, x) \in ]0, T[ \times \R$,
provided $(\rho, u)$ denotes the solution of the Cauchy problem~\eqref{e:nlGARZ},\eqref{e:id}, $\xi$ is as in~\eqref{e:nlGARZ}  and $T = C(V, M)$. 
\end{corol}
\begin{proof}
We can argue as in \S\ref{sss:limite}, the only new point is establishing the \emph{pointwise} convergence $\xi_k (t, x) \to \xi (t, x)$. Towards this end, it suffices to recall Remark~\eqref{r:dafermos}, which gives a meaning to $\rho_k(t, \cdot)$ for \emph{every} $t$, and then apply the same argument as in {Step 2} of the proof of Theorem 1.1, item (i) in~\cite{ColomboCrippaSpinolo2}, which yields $\rho_k (t, \cdot) \weaks \rho(t, \cdot)$ for \emph{every} $t$.  
\end{proof}

\section{Oleinik estimate for a general function $V$}\label{s:olenik}
This section is devoted to the proof of Theorem \ref{t:oleinik}. We first prove the following one-sided Lipschitz estimate, uniform in the parameter $\ee$ as $\ee\to 0^+$. 
\begin{proposition}\label{p:oleinik}
Under the same assumptions as in the statement of Theorem~\ref{t:oleinik} there are positive constants $\ee_0, a, c>0$ only depending on $V,\|z_0\|_{L^\infty},\|\psi_0\|_{L^\infty}$ such that 
\begin{equation}\label{e:onesidexi}
\xi_\ee(t,y)-\xi_\ee(t,x)\ge -\left(a + \frac1{ct}\right)(y-x) \quad \text{for every $t>0$, $\ee \in ]0, \ee_0[$ and $x<y$}.
\end{equation}
\end{proposition}
\begin{proof}
To simplify the notation, in the proof of Proposition~\ref{p:oleinik} we write $\xi$, $\rho$, $u$ rather than $\xi_\ee$, $\rho_\ee$, $u_\ee$. We also set $h: = \partial_x \xi_\ee$. 
Owing to the stability result provided by Corollary~\ref{c:stability} we can assume with no loss of generality that the initial data $(\rho_0, u_0)$ are smooth and that $\rho_0$ is compactly supported, which implies that $\rho(t,\cdot)$ is compactly supported for every $t>0$ and hence that $\lim_{|x|\to \infty}h(t,x)=0$ for every $t\ge 0$. Owing to the propagation of regularity property~\eqref{e:est_rho}, $h$ is Lipschitz continuous. We set 
\be \label{e:emme0}
    m(t) : = \inf_{y \in \R}  h(t, y).  
\eq 
Since $h(t, \cdot)$ vanish at $\pm \infty$, for every $t\ge 0$ either  $h(t, \cdot) \ge 0$ and $m(t)=0$ or $m(t) <0$ and there is $x \in \R$ such that 
\be \label{e:emme}
    m(t) = h(t,x)= \min_{y \in \R}  h(t, y).  
\eq
Note furthermore that $m$ is a Lipschitz continuous function being the infimum of Lipschitz continuous functions, and as such differentiable at a.e. $t \in \R_+$. 
Assume that we have proved that for some suitable constants $c_0, c_1, c_2>0$ we have the inequality 
\be \label{e:goalolenik}
m'(t) \ge c_2 m^2(t) + c_1 m(t) - c_0 \quad \text{for a.e. $t \ge 0$}.
\eq
We claim that this implies 
\be \label{e:goalolenik2}
m(t) \ge - a - \frac{1}{ct}
\eq
for suitable constants $a>0$ and $c>0$. Note that~\eqref{e:goalolenik2} immediately yields~\eqref{e:onesidexi}. To see that~\eqref{e:goalolenik} implies~\eqref{e:goalolenik2} we first of all point out that the 
polynomial $c_2 x^2 + c_1x - c_0$ has two real and distinct roots, $x=m_1>0$ and $x=-m_2<0$.  By setting $a: = 2 m_2$, we can find $c \in ]0, c_2[$ in such a way that 
$$
    c_2 x^2 + c_1x - c_0 \ge c \, x^2 \quad \text{if $x \in ]-\infty, - a[$}. 
$$ 
This implies that, if $m(t) \leq -a$, then $m'(t) \ge c m^2 (t)$, which by the Comparison Theorem for ODEs yields $m(t) \ge - \sfrac{1}{ct}$. This in turn shows that~\eqref{e:goalolenik} implies~\eqref{e:goalolenik2}. 

Wrapping up, to establish~\eqref{e:onesidexi} we are left to prove~\eqref{e:goalolenik}. Note furthermore that, since $m(t) \leq 0$ for every $t\ge 0$, at any point $t$ of differentiability for $m$ where $m(t)=0$ we have $m'(t)=0$, which obviously implies~\eqref{e:goalolenik}. We are therefore left to verify~\eqref{e:goalolenik} at a.e. time $t$ where~\eqref{e:emme} holds. Towards this end, it suffices to show that 
\be \label{e:oleinik3}
      \frac{\partial h}{\partial t} (t, x)  \ge c_2 m^2(t) + c_1 m(t) - c_0 
\eq
provided $x$ is the same as in~\eqref{e:emme}.  Before entering the details, we make two last preliminary comments: first, 
in the following, we use the algebraic identity
\be
\label{e:magic}
      \rho= \xi - \ee h,
\eq
which follows from the fact that the kernel is the exponential function. Second,
in the rest of the proof we always assume that~\eqref{e:boundrho},
which can be justified by relying on~\eqref{e:localex} and a continuous induction argument. Indeed,~\eqref{e:localex} implies that $\| \rho(t, \cdot) \|_{L^\infty} \leq \sfrac{3}{2}$ for every $t \in [0, \tau]$, for a sufficiently small $\tau$ not depending on $\ee$. In the following, we will show that~\eqref{e:boundrho} implies~\eqref{e:onesidexi}, which in turn owing to~\eqref{e:magic} implies that on $[\tau, + \infty[$
$$
   \| \rho (t, \cdot) \|_{L^\infty} \leq 1 + \ee \left(a + \frac1{ct}\right) \leq 1 + \ee \left(a + \frac1{c \tau}\right)  \leq \frac{3}{2}
$$
provided $\ee$ is small enough, which obviously implies~\eqref{e:boundrho}. The rest of the proof is organized in the following steps. \\
{\bf Step 1} (equation for $h$). 
By convolving \eqref{e:nlGARZ} with $\eta_\ee$, we have 
$$
    \partial_t \xi + V (\xi, u) \partial_x \xi = - \frac{1}{\ee} \int_x^{+ \infty} 
    \exp \left( \frac{x-y}{\ee}\right) \partial_y V (\xi(y), u(y)) \xi(y) dy, 
$$
and by $x$-differentiating we get 
\begin{equation*}
\begin{split}
     \partial_t h + V(\xi, u) \partial_x h &  =- \partial_x [V (\xi, u)] h + \frac{1}{\ee}
     \left[ 
\partial_x [V (\xi, u)]
    \xi (x) 
 - \frac{1}{\ee} \int_x^{+ \infty} 
    \exp \left( \frac{x-y}{\ee}\right) \partial_y[ V (\xi(y), u(y))] \xi(y) dy
    \right]
\end{split}
\end{equation*}
We now evaluate the previous expression at a minimum point (where $\partial_x h$ vanishes) and by recalling that $\partial_x u= \rho  z $
we get 
$$
    \partial_x [V (\xi, u)] = \partial_1 V(\xi, u)  h + \partial_2 V (\xi, u) \rho  z.
$$
All in all, we arrive at 
\begin{equation} \label{e:deracca}
\begin{split}
     \partial_t h &   =  - \partial_1   V(t, x) h^2 (t, x) -  \partial_2 V  (t, x) \rho(t, x) h(t, x) z(t, x) 
    \\ &  + \underbrace{
    \frac{1}{\ee}
     \left[ 
  \partial_1 V (t, x) h (t, x)
    \xi (t, x) 
 - \frac{1}{\ee} \int_x^{+ \infty} 
    \exp \left( \frac{x-y}{\ee}\right)  \partial_1 V (t, y) h (t, y)
    \xi(t, y)  dy
    \right]}_{: = T_A} \\ & 
   + \underbrace{\frac{1}{\ee}
     \left[ 
    \partial_2 V (t, x)  \rho (t, x) z(t, x)
    \xi (t, x) 
 - \frac{1}{\ee} \int_x^{+ \infty} 
    \exp \left( \frac{x-y}{\ee}\right) \partial_2 V (t, y)  \rho (t, y) z(t, y) \xi(t, y) dy
    \right].}_{: = T_B}
\end{split}
\end{equation}
In the previous expression and in the following sometimes to simplify he notation we sometimes write $V (t, x)$ rather than $V(\xi(t, x), u(t, x))$.\\
{\bf Step 2} (estimate of $T_A$). We have
\begin{equation}\label{e:uno}
\begin{split}
    T_A  & =  
     \frac{\partial_1 V (t, x) h (t, x)
    \xi (t, x)  }{\ee} -   
  \frac{ \partial_1 V (t, x) 
    \xi (t, x)}{\ee^2} \int_x^{+ \infty} 
    \exp \left( \frac{x-y}{\ee}\right) h (t, y) dy
     \\ & \quad  +    
   \underbrace{ \frac{1}{\ee^2} \int_x^{+ \infty} 
    \exp \left( \frac{x-y}{\ee}\right) \Big[  \xi(t, x) \partial_1 V (\xi(t, x), u(t, x)) 
    -  \xi(t, y) \partial_1 V (\xi(t, y), u(t, x)) 
      \Big] h(t, y) dy}_{: = T_{A1}} \\ &     
\quad + \underbrace{ \frac{1}{\ee^2} \int_x^{+ \infty}
    \exp \left( \frac{x-y}{\ee}\right) \Big[\partial_1 V (\xi(t, y), u(t, x))   - \partial_1 V (\xi(t, y), u(t, y))  \Big] \xi(t, y) h(t, y)dy }_{: = T_{A2}}  \\
\end{split}
\end{equation}
and 
\begin{equation} \label{e:TA1}
\begin{split}
    T_{A1}    & =     
    - \frac{\partial_1 V (\xi(t, x), u(t, x))}{\ee^2} \int_x^{+ \infty} 
    \exp \left( \frac{x-y}{\ee}\right) [  \xi(t, y)  
    -  \xi(t, x) 
      ] h(t, y) dy \\ &     
     \qquad  + \underbrace{ \frac{1}{\ee^2} \int_x^{+ \infty}
       \exp \left( \frac{x-y}{\ee}\right) \Big[\partial_1 V (\xi(t, x), u(t, x))   - \partial_1 V (\xi(t, y), u(t, x))  \Big] \xi(t, y) h(t, y)dy}_{: = T_{A12}}
\end{split}
\end{equation}
We now argue as in~\cite{8autori}, use~\eqref{e:magic} and arrive at 
\begin{equation} \label{e:xixxiy}
\begin{split}
    \frac{1}{\ee^2} \int_x^{+ \infty} &
    \exp \left( \frac{x-y}{\ee}\right) [\xi(t, y) - \xi(t, x) ]h(t, y)dy =
    \frac{1}{\ee^2} \int_x^{+ \infty}
    \exp \left( \frac{x-y}{\ee}\right) \rho(t, y) h(t, y)dy \\& \quad + \frac{1}{\ee} \int_x^{+ \infty} 
    \exp \left( \frac{x-y}{\ee}\right) h^2(t, y)dy - \xi(t, x)  \frac{1}{\ee^2} \int_x^{+ \infty}
    \exp \left( \frac{x-y}{\ee}\right) h(t, y)dy
\end{split}
\end{equation}
Owing to~\eqref{e:emme} we have
\be \label{e:accax}
    \frac{1}{\ee^2} \int_x^{+ \infty}
    \exp \left( \frac{x-y}{\ee}\right) \rho(t, y) h(t, y)dy \ge h(t, x) \frac{1}{\ee^2} \int_x^{+ \infty}
    \exp \left( \frac{x-y}{\ee}\right) \rho(t, y) dy = \frac{h(t, x) \xi(t, x)}{\ee}
\eq
and by combining the above estimates we conclude that
\be \label{e:uno2}
\begin{split}
    T_{A} \ge   & 
     \frac{\partial_1 V (t, x) h (t, x)
    \xi (t, x)  }{\ee} -   
  \frac{ \partial_1 V (t, x) 
    \xi (t, x)}{\ee^2} \int_x^{+ \infty} 
    \exp \left( \frac{x-y}{\ee}\right) h (t, y) dy
      +   T_{A1} + T_{A2} \\
&  \stackrel{\eqref{e:TA1},\eqref{e:xixxiy},\eqref{e:accax}}{\ge} 
- \frac{\partial_1 V (\xi(t, x), u(t, x))}{\ee} \int_x^{+ \infty} 
    \exp \left( \frac{x-y}{\ee}\right) h^2(t, y)dy + T_{A12} + T_{A2} 
\end{split}
\eq  
Next, we point out that 
\begin{equation}
\label{e:t4} \begin{split}
  |T_{A2}| & \stackrel{\eqref{e:civ}}{\leq}
  \frac{ C(V)}{\ee^2}
  \int_x^{+ \infty}
    \exp \left( \frac{x-y}{\ee}\right) | u(t, x) - u(t, y)  | \xi(t, y) | h(t, y)| dy \\ &
  \stackrel{\eqref{e:boundrho},\eqref{e:mpxi},\eqref{e:mpii}}{\leq}
     \frac{C(V) \| z_0 \|_{L^\infty}}{\ee^2}
    \int_x^{+ \infty} \exp \left( \frac{x-y}{\ee}\right) [y-x] | h(t, y)| dy.
\end{split}
\end{equation} 
We now control $T_{A12}$. First, we decompose it as 
\be \label{e:TA12} \begin{split}
    T_{A12} & =
    \underbrace{\frac{1}{\ee^2} \int_x^{+ \infty}
       \exp \left( \frac{x-y}{\ee}\right) \Big[\partial_1 V (\xi(t, x), u(t, x))   - \partial_1 V (\xi(t, y), u(t, x))  \Big] \xi(t, y) h(t, y) \mathbbm{1}_{h(t, y) \ge 0}dy}_{: = T_{A121}} \\
     & + \underbrace{\frac{1}{\ee^2} \int_x^{+ \infty}
       \exp \left( \frac{x-y}{\ee}\right) \Big[\partial_1 V (\xi(t, x), u(t, x))   - \partial_1 V (\xi(t, y), u(t, x))  \Big] \xi(t, y) h(t, y) \mathbbm{1}_{h(t, y) \leq 0}dy}_{:= T_{A122}}
\end{split}
\eq
To control $T_{A121}$, we use the Fundamental Theorem of Calculus and write 
\be \label{e:tfc} \begin{split}
    \partial_1 V (\xi(t, x), u(t, x))   & - \partial_1 V (\xi(t, y), u(t, x))  = \int_{x}^y [-\partial_{11} V (\xi(t, u), u(t, x))]
    h(t, u) du  \\
   & \stackrel{\eqref{e:V2},\eqref{e:emme}}{\ge} \beta_V m(t) [y-x].
\end{split}\eq
Since we are handling the case $m(t)<0$, this yields 
\be \label{e:TA121} \begin{split}
    T_{A121} & \ge 
    \frac{\beta_V m(t)}{\ee} \int_x^{+ \infty}
       \exp \left( \frac{x-y}{\ee}\right)  \left[ \frac{y-x}{\ee}\right] \xi(t, y) h(t, y) \mathbbm{1}_{h(t, y) \ge 0}dy\\
& 
    \stackrel{m<0, 0 \leq \xi \leq 1}{\ge}  \frac{\beta_V m(t)}{\ee} \int_x^{+ \infty}
       \exp \left( \frac{x-y}{\ee}\right)  \left[ \frac{y-x}{\ee}\right] | h(t, y)| dy.
\end{split}\eq
To control $T_{A122}$ we use again~\eqref{e:tfc} and introduce the decomposition 
\[ \begin{split}
     \int_{x}^y [-\partial_{11} V (\xi(t, w), u(t, x))]
    h(t, w) dw & =  \int_{x}^y [-\partial_{11} V (\xi(t, w), u(t, x))]
    h(t, w) \mathbbm{1}_{h(t, w) \ge 0}dw \\
    & \quad +   \int_{x}^y [-\partial_{11} V (\xi(t, w), u(t, x))]
    h(t, w) \mathbbm{1}_{h(t, w) \leq 0}dw,
\end{split} \] 
\be \label{e:TA122}\begin{split}
  T_{A122} & =
  \frac{1}{\ee^2} \int_x^{+ \infty}
       \exp \left( \frac{x-y}{\ee}\right) \xi(t, y) h(t, y) \mathbbm{1}_{h(t, y) \leq 0} \int_{x}^y [-\partial_{11} V (\xi(t, w), u(t, x))]
    h(t, w) \mathbbm{1}_{h(t, w) \ge 0} dw dy \\ & +
    \frac{1}{\ee^2} \int_x^{+ \infty}
       \exp \left( \frac{x-y}{\ee}\right) \xi(t, y) h(t, y) \mathbbm{1}_{h(t, y) \leq 0} \int_{x}^y \underbrace{[-\partial_{11} V (\xi(t, w), u(t, x))]}_{\ge 0}
    h(t, w) \mathbbm{1}_{h(t, w) \leq  0} dw dy \\
   &  \stackrel{ \xi \ge 0}{\ge}
    \frac{1}{\ee^2} \int_x^{+ \infty}
       \exp \left( \frac{x-y}{\ee}\right) \xi(t, y) h(t, y) \mathbbm{1}_{h(t, y) \leq 0} \int_{x}^y [-\partial_{11} V (\xi(t, w), u(t, x))]
    h(t, w) \mathbbm{1}_{h(t, w) \ge 0} dw dy \\
   & \stackrel{\eqref{e:emme}, m <0}{\ge}
     \frac{m(t)}{\ee^2} \int_x^{+ \infty}
       \exp \left( \frac{x-y}{\ee}\right) \xi(t, y)  \int_{x}^y [-\partial_{11} V (\xi(t, w), w(t, x))]
    h(t, w) \mathbbm{1}_{h(t, w) \ge 0} dw dy \\ & =
      \frac{m(t)}{\ee^2} \int_x^{+ \infty} 
       [-\partial_{11} V (\xi(t, w), u(t, x))]
    h(t, w) \mathbbm{1}_{\{h(t, w) \ge 0\}}  \int_{u}^{+\infty}  \exp \left( \frac{x-y}{\ee}\right) \xi(t, y)dy dw 
   \\ &  \stackrel{0 \leq \xi \leq 1, m <0}{\ge}
     \frac{m(t)}{\ee} \int_x^{+ \infty} 
       \exp \left( \frac{x-w}{\ee}\right)  [-\partial_{11} V (\xi(t, w), u(t, x))]
    h(t, w) \mathbbm{1}_{\{h(t, w) \ge 0\}}  dy dw \\
    & \stackrel{\eqref{e:V2}}{\ge}
    \frac{\beta_V m(t)}{\ee} \int_x^{+ \infty} 
       \exp \left( \frac{x-w}{\ee}\right)  
   | h(t, w)|    dw
\end{split}
\eq
By plugging~\eqref{e:t4},\eqref{e:TA12},\eqref{e:TA121} and~\eqref{e:TA122} into~\eqref{e:uno2} and recalling~\eqref{e:V1} we get
\begin{equation}\label{e:estTAfirst}
\begin{split}
T_A & 
\ge  \frac{1}{\ee} \int_x^{+ \infty} 
    \exp \left( \frac{x-y}{\ee} \right) \left[ \frac{\alpha_V}{6} h^2(t, y) - C(V) \|z\|_{L^\infty}  |h(t, y) | \frac{y-x}{\ee} \right] dy   
    \\ & +  
    \frac{1}{\ee} \int_x^{+ \infty} 
    \exp \left( \frac{x-y}{\ee} \right) \left[ \frac{\alpha_V}{6} h^2(t, y) + 2 \beta_V m(t) |h(t, y)| \frac{y-x}{\ee}   \right] dy \\
     & + 
     \frac{1}{\ee} \int_x^{+ \infty} 
    \exp \left( \frac{x-y}{\ee} \right) \left[ \frac{\alpha_V}{6} h^2(t, y) + 2 \beta_V m(t) |h(t, y)| \right] dy \\
    & + 
    \frac{1}{\ee} \int_x^{+ \infty}  \exp \left( \frac{x-y}{\ee} \right) \frac{\alpha_V}{2} h^2(t, y) dy.
\end{split}
\end{equation}
We now use the elementary inequality
\be \label{e:binomio}
    h^2 - c |h| \ge - \frac{c^2}{4}, \quad \text{for every $h\in \R$, $c>0$},
\eq 
 recall the with identities
\be \label{e:integrbyparts}
    \frac{1}{\ee} \int_x^{+ \infty} 
    \exp \left( \frac{x-y}{\ee} \right) \frac{[y-x]}{\ee} dy=1, \qquad 
     \frac{1}{\ee} \int_x^{+ \infty} 
    \exp \left( \frac{x-y}{\ee} \right) \frac{[y-x]^2}{\ee^2} dy=2
\eq  and control the first line in~\eqref{e:estTAfirst}:
\begin{equation*}
\begin{split}
     \frac{1}{\ee} & \frac{\alpha_V}{6}  \int_x^{+ \infty} 
    \exp \left( \frac{x-y}{\ee} \right) \left[ h^2(t, y) - \frac{C(V)}{\alpha_V} \|z_0\|_{L^\infty}  |h(t, y) | \frac{y-x}{\ee} \right] dy  \\
   & =
    \frac{1}{\ee} \frac{\alpha_V}{6}  \int_x^{ x+ \ee} 
    \exp \left( \frac{x-y}{\ee} \right) 
   \left[ h^2(t, y) - \frac{C(V)}{\alpha_V} \|z_0\|_{L^\infty} |h(t, y) | \frac{y-x}{\ee} \right] dy
 \\ & \quad + 
 \frac{1}{\ee} \frac{\alpha_V}{6}  \int_{ x+ \ee}^{+ \infty}
    \exp \left( \frac{x-y}{\ee} \right) 
   \left[ h^2(t, y) - \frac{C(V)}{\alpha_V} \|z_0\|_{L^\infty}  |h(t, y) | \frac{y-x}{\ee} \right] dy
   \\ & \ge 
       \frac{1}{\ee} \frac{\alpha_V}{6}  \int_x^{ x+ \ee} 
    \exp \left( \frac{x-y}{\ee} \right) \frac{y-x}{\ee}
   \left[ h^2(t, y) - \frac{C(V)}{\alpha_V} \|z_0\|_{L^\infty}|h(t, y) |  \right] dy
   \\ & \quad + \frac{1}{\ee} \frac{\alpha_V}{6}  \int_{ x+ \ee}^{+ \infty}
    \exp \left( \frac{x-y}{\ee} \right) 
   \left[ h^2(t, y) -\frac{C(V)}{\alpha_V} \|z_0\|_{L^\infty}|h(t, y) |  \right] dy\\
   & \stackrel{\eqref{e:binomio},\eqref{e:integrbyparts}}{\ge} 
    - \frac{C(V) \|z_0\|^2_{L^\infty}}{\alpha_V},
\end{split}
\end{equation*}
where we recall that the exact value of the constants can vary from line to line. 
We use analogous computations to control the second and third line of~\eqref{e:estTAfirst} and eventually arrive at

\be \label{e:TA}
\begin{split}
T_A \ge &~  -\frac{C(V) \|z_0\|^2_{L^\infty}}{\alpha_V} - \frac{54\beta_V^2 m(t)^2}{\alpha_V}+ 
    \frac{\alpha_V}{2} \frac{1}{\ee} \int_x^{+ \infty}  \exp \left( \frac{x-y}{\ee} \right)  |h(t, y)|^2 dy. 
\end{split}
\eq
{\bf Step 2} (estimate of $T_B$).
First, we point out that the term $T_B$ has the following structure:
\[
T_B = g \ast \left(\frac1\ee \delta_0 - \frac1\ee \eta_\ee \right)= g\ast (-\eta_\ee') =-\partial_x g\ast \eta_\ee = -\frac1\ee\int_x^{+\infty}\exp \left( \frac{x-y}{\ee} \right) \partial_y g(t, y) dy,
\]
provided $g= \partial_2V \partial_xu \  \xi$.
Since
\[
\partial_y g = \partial^2_{22}V [\partial_y u]^2 \xi + \partial^2_{12}V h \partial_xu \ \xi + \partial_2V \partial^2_{yy}u \ \xi + \partial_2 V \partial_y u h
\]
then
\[
\begin{split}
T_B = &~ -\frac1\ee\int_x^{+\infty}\exp \left( \frac{x-y}{\ee} \right) \left[\partial^2_{22}V [\partial_y u]^2 \xi + \partial^2_{12}V h \partial_xu \ \xi + \partial_2V \partial^2_{yy}u \ \xi + \partial_2 V \partial_y u h \right](t, y) dy \\
\doteq &~T_{B1}+ T_{B2} + T_{B3} + T_{B4}.
\end{split}
\]
Owing to~\eqref{e:mpii} and ~\eqref{e:magic} we have $\partial_x u = \rho z = (\xi - \ee h) z$. Owing to~\eqref{e:mpxi} and again to~\eqref{e:mpii} this implies 
\begin{equation*}\label{e:TB134}
\begin{split}
|T_{B1}| &\le C(V) \|z_0\|_{L^\infty}^2 + 2\ee  C(V)  \|z_0\|^2_{L^\infty} \frac1\ee \int_x^{+\infty}\! \! \! \exp \left( \frac{x-y}{\ee} \right)  |h (t, y)| dy \\
& + \ee^2  C(V)  \|z_0\|^2_{L^\infty} \frac1\ee \int_x^{+\infty} \! \! \! \exp \left( \frac{x-y}{\ee} \right)  |h (t, y)|^2 dy
\end{split}
\end{equation*}
and 
\[
|T_{B2}|+|T_{B4}| \le 2 C(V)  \|z_0 \|_{L^\infty} \frac1\ee \int_x^{+\infty} \! \! \! \! \exp \left( \frac{x-y}{\ee} \right)  |h (t, y)| dy + 
2 \ee  C(V)  \|z_0\|_{L^\infty} \frac1\ee \int_x^{+\infty} \! \! \!  \!\exp \left( \frac{x-y}{\ee} \right)  |h (t, y)|^2 dy.
\]
We are left to control $T_{B3}$: we use~\eqref{e:mpii} and~\eqref{e:est_h} and point out that  
since $\partial_xu=\rho z$ then $\partial^2_{xx}u= \partial_x \rho z + \rho \partial_x z =  \partial_x \rho z + \rho^2 \psi $.  We then decompose $T_{B3}$ as 
\[
T_{B3}= -\frac1\ee \int_x^{+\infty}\exp \left( \frac{x-y}{\ee} \right) \partial_2 V  \left(\partial_x \rho z + \rho^2 \psi\right)\xi dy \doteq T_{B31}+T_{B32}.
\]
We first control $|T_{B32}|$ and by using~\eqref{e:mpxi},\eqref{e:est_h} and~\eqref{e:magic} we arrive at
\[
\begin{split}
|T_{B32}| \le &~ C(V)\|\psi_0\|_{L^\infty} + 2 \ee C(V)\|\psi_0\|_{L^\infty}  \frac1\ee \int_x^{+\infty}\exp \left( \frac{x-y}{\ee} \right) |h(t, y)| dy\\
&~ + \ee^2 C(V)\|\psi_0\|_{L^\infty}  \frac1\ee \int_x^{+\infty}\exp \left( \frac{x-y}{\ee} \right) |h(t, y)|^2 dy.
\end{split}
\]
We now focus on $T_{B31}$: by differentiating~\eqref{e:magic} we get $\partial_x \rho = h - \ee \partial_x h$, which yields 
\[
T_{B31} = -\frac1\ee \int_x^{+\infty}\exp \left( \frac{x-y}{\ee} \right)  \partial_2 V  \left( z h \xi - z \ee \xi \partial_x h\right) dy \doteq T_{B311}+ T_{B312}.
\]
Owing to~\eqref{e:mpxi},\eqref{e:est_h} we get 
\[
|T_{B311}| \le C(V) \|z_0\|_{L^\infty}  \int_x^{+\infty}\frac1\ee \exp \left( \frac{x-y}{\ee} \right) |h(t, y)| dy
\]
To control $T_{B312}$ we use the Integration by Parts Formula and arrive at 
\[
T_{B312}= -  [\partial_2 V z \,  \xi \,  h ](t, x) + \int_x^{+\infty} \partial_y \left[\exp \left( \frac{x-y}{\ee} \right)  \partial_2 V z \, \xi\right] h  dy \doteq - \partial_2 V (t,x) z(t,x)\xi(t,x)m(t)  + T_C.
\]
We have 
\[
\begin{split}
T_C = ~ \int_x^{+\infty} & \! \! \left[  -\frac1\ee \exp \left( \frac{x-y}{\ee} \right)  \partial_2V z \, \xi h + \exp \left( \frac{x-y}{\ee} \right)  \partial_y[\partial_2 V] z \xi h \right. \\
& + \left. \exp \left( \frac{x-y}{\ee} \right)  \partial_2 V \partial_x z \xi h + \exp \left( \frac{x-y}{\ee} \right) \partial_2 V z h^2\right] dy \\
\doteq & ~ T_{C1} + T_{C2} + T_{C3} + T_{C4}. \phantom{\int}
\end{split}
\] 
Owing to~\eqref{e:mpxi},\eqref{e:est_h} we have 
\[
\begin{split}
|T_{C1}| \le& ~  C(V) \|z_0 \|_{L^\infty} \int_x^{+\infty} \frac1\ee \exp \left( \frac{x-y}{\ee} \right)  |h(t, y)| dy, \\
|T_{C4}| \le& ~  \ee C(V) \|z_0\|_{L^\infty} \int_x^{+\infty} \frac1\ee \exp \left( \frac{x-y}{\ee} \right)  |h(t, y)|^2 dy.
\end{split}
\]
By using again the identity $\partial_x z = \rho \psi = (\xi - \ee h) \psi$, we get 
\[
|T_{C3}| \le \ee C(V)\|\psi_0\|_{L^\infty} \int_x^{+\infty} \frac1\ee \exp \left( \frac{x-y}{\ee} \right)  |h(t, y)| dy + \ee^2 C(V)\|\psi_0\|_{L^\infty} \int_x^{+\infty} \frac1\ee \exp \left( \frac{x-y}{\ee} \right)  |h(t, y)|^2 dy.
\] 
We also have 
\[
T_{C2}= \int_x^{+\infty} \exp \left( \frac{x-y}{\ee} \right)  \left[(\partial^2_{22} V)\partial_x u z \xi h +  (\partial^2_{21} V) z \xi h^2\right] dy \doteq T_{C21}+ T_{C22}
\]
and, since $\partial_x u = \rho z = (\xi - \ee h)z$ and by recalling~\eqref{e:mpxi} and~\eqref{e:mpii}, this implies 
\[
\begin{split}
|T_{C21}| \le &~ C(V) \|z_0 \|_{L^\infty}^2 \ee \int_x^{+\infty}\frac1\ee \exp \left( \frac{x-y}{\ee} \right) |h(t, y)| dy +  C(V) \|z_0\|_{L^\infty}^2 \ee^2 \int_x^{+\infty}\frac1\ee \exp \left( \frac{x-y}{\ee} \right) |h(t, y)|^2 dy,  \\
|T_{C22}|\le &~  C(V) \|z_0\|_{L^\infty} \ee \int_x^{+\infty}\frac1\ee \exp \left( \frac{x-y}{\ee} \right) |h(t, y)|^2 dy.
\end{split}
\]
We combine all the above estimates and recall that 
 $m(t)<0$ to arrive at 
\be \label{e:TB}
\begin{split}
|T_B| \le &~ |T_{B1}| + |T_{B2}| + |T_{B4}| + |T_{B3}| \\
\le &~ |T_{B1}| + |T_{B2}| + |T_{B4}| + |T_{B32}| + |T_{B311}| + |\partial_2 V| \|z_0 \|_{L^\infty} |m(t)| +
|T_{C1}| + |T_{C3}|+|T_{C4}| \\ & + |T_{C21}|+|T_{C22}|  \leq~ -\tilde c m(t)  + \tilde c_0 + \tilde c_1(|h|\ast \eta_\ee)(x) +  \ee  \tilde c_2(h^2\ast \eta_\ee)(x),
%
\end{split}
\eq
where we have used the inequality $\ee \leq 1$, the shorthand notation 
\be
|h|\ast \eta_\ee(x) = \int_x^{+\infty}\frac1\ee \exp \left( \frac{x-y}{\ee} \right) |h(t, y)| dy, \quad h^2 \ast \eta_\ee(x) = \int_x^{+\infty}\frac1\ee \exp \left( \frac{x-y}{\ee} \right) |h(t, y)|^2 dy  
 \eq
 and 
\[
\begin{split}
& \tilde c =  C(V)\|z_0\|_{L^\infty}, \quad \tilde c_0 = C(V) \left[ \|z_0\|_{L^\infty}^2 + \|\psi_0\|_{L^\infty} \right], \quad\tilde c_1 = C(V) \left[ \|z_0\|_{L^\infty} + \|z_0\|_{L^\infty}^2 +  \|\psi_0\|_{L^\infty}\right] \\
& \tilde c_2 =    C(V)  \left[ \|z_0\|_{L^\infty} +  \|z_0\|_{L^\infty}^2 +  \|\psi_0\|_{L^\infty} \right],
\end{split}
\]
and the fact that the exact value of $C(V)$ can vary from occurrence to occurrence. \\
{\bf Step 3} (conclusion).
We now plug \eqref{e:TA} and \eqref{e:TB} into \eqref{e:deracca}, use~\eqref{e:magic} and~\eqref{e:mpii} and conclude that at any point of differentiability of $m$ where $m<0$ we have 
\[
\begin{split}
\frac{dm}{dt} &    =  - \partial_1   V(t, x) m^2 (t) -  \partial_2 V  (t, x) \rho(t, x) m(t) z(t, x) +T_A + T_B \\
    \ge &~ \left[ \alpha_V - \frac{54 \beta_V^2}{\alpha_V} - \ee \|z_0\|_{L^\infty}C(V) \right] m^2(t) + \left[ C(V) \|z_0\|_{L^\infty} + \tilde c \right] m(t) - 
       \left[ \tilde c_0 + \frac{C(V) \| z_0 \|^2_{L^\infty}}{\alpha_V}\right] \\
    & ~  + \left[ \frac{\alpha_V}2 - \ee \tilde c_2 \right] (h^2 \ast \eta_\ee)(x) - \tilde c_1 (|h|\ast \eta_\ee)(x)
\end{split}
\]
We point out that for  $\ee$ sufficiently small we have $\sfrac{\alpha_V}{2} - \ee \tilde c_2>\sfrac{\alpha_V}{4}$, we use~\eqref{e:binomio} to conclude 
\[
\left[ \frac{\alpha_V}2 - \ee \tilde c_2 \right] (|h|^2 \ast \eta_\ee)(x) - \tilde c_1 (|h|\ast \eta_\ee)(x) \ge - \frac{\tilde c_1^2}{2\alpha_V}.
\]
This eventually implies that $m$ satisfies~\eqref{e:goalolenik} 
provided $\ee$ is small enough and 
\[
\begin{split}
c_2 = &~ \alpha_V - \frac{54 \beta_V^2}{\alpha_V} -\ee \|z_0\|_{L^\infty}C(V)   \qquad 
c_1 =  C(V)\|z_0\|_{L^\infty} + \tilde c \\
c_0 = &~ \tilde c_0 + \frac{C(V) \| z_0 \|^2_{L^\infty}}{\alpha_V} + \frac{\tilde c_1^2}{2\alpha_V}.  
\end{split}
\]
Note that, under~\eqref{e:V3}, $c_2 >0$ for $\ee$ small enough. This concludes the proof of Proposition~\ref{p:oleinik}. 
\end{proof}
\begin{remark}
As pointed out at the beginning of the proof of Proposition~\ref{p:oleinik}, the differential inequality~\eqref{e:goalolenik} implies $m'(t) \ge c m^2(t)$ for a suitable constant $c>0$ as soon as $m (t) <-a$ for some positive constant $a>0$. 
If the initial datum $\rho_0$ satisfies the  one-sided Lipschitz condition
$$
    \rho_0(y)-\rho_0(x) \ge m_0 [y-x]\qquad \text{for every $ x< y$ and for some $m_0<0$}
$$
then we have 
\[
\xi(t,y)-\xi(t,x) \ge {\color{black} \left[-a + \frac1{\sfrac1{m_0}-c t}\right] }(y-x) \qquad \text{for every $ x< y$ and $t>0$}.
\]
{\color{black}The above estimate is better than \eqref{e:onesidexi}, since the one-side Lipschitz bound does not blow-up as $t\to 0$.}
\end{remark}

\begin{proof}[Proof of Theorem \ref{t:oleinik}] We proceed according to the following steps. \\
{\bf Step 1} (compactness of $\{ \xi_\ee \}_{\ee >0}$ and $\{ \rho_\ee \}_{\ee >0}$). By combining \eqref{e:mpxi} and \eqref{e:onesidexi} we conclude that for every $t,L>0$ we have 
\be \label{e:tvxi}
\mathrm{Tot.Var.} \, \{ \xi_\ee(t,\cdot); [-L,L]\} \le 4L \left( \alpha_0 + \frac1{ct}\right) + 1
\eq
and that the distribution $h_\ee\doteq \partial_x \xi_\ee$ satisfies 
\be \label{e:onesidedLip}
h_\ee\doteq \partial_x \xi_\ee \ge -\left(\alpha_0 + \frac1{ct} \right) \quad \text{for every $t>0$}
\eq
and every $\ee$ small enough, which in particular implies that $h_\ee$ is actually a measure. We also recall~\eqref{e:boundrho} and conclude that, for every vanishing sequence $\ee_k\to 0^+$ there are a subsequence (which we do not relabel) and a bounded function $\bar \rho$ such that $\rho_{\ee_k}\rightharpoonup^\ast \bar \rho$ weakly$^\ast$ in $L^\infty(\R^+\times \R)$. Note that actually a slightly stronger result holds:  recalling Remark~\ref{r:dafermos}, we have $ \rho_{\ee_k} (t, \cdot) \rightharpoonup^\ast \rho(t, \cdot)$ for \emph{every} $t\ge 0$. This follows from very the same argument as in {\sc Step 2} of the proof of Theorem 1.1, item (i) in~\cite{ColomboCrippaSpinolo2} and implies that 
\begin{equation}\label{e:Linfglobal2}
          \xi_{\ee_k} (t, \cdot) = \rho_{\ee_k} (t, \cdot) \ast \eta_{\ee_k} \weaks \rho (t, \cdot)  \quad \text{in the space of distributions $\mathcal D'(\R)$.}
 \end{equation}
We now consider the sequence $\{ \xi_{\ee_k} \}_{k \in \mathbb N}$: by combining~\eqref{e:mpxi} with~\eqref{e:tvxi} and the Helly-Frech\'et-Kolmogorov Compactness Theorem we conclude that, for every $t>0$, the sequence $\{ \xi_{\ee_k} (t, \cdot) \}_{k \in \mathbb N}$ converges (in principle up to subsequences) in the strong topology of $L^1_{\mathrm{loc}} (\R)$  to some limit function $\tilde \rho$. Owing to~\eqref{e:Linfglobal2}, we have $\tilde \rho = \bar \rho (t, \cdot)$ and by the  Urysohn Subsequence Principle we conclude that 
\begin{equation}\label{e:Linfglobal3}
          \xi_{\ee_k} (t, \cdot)\to \bar  \rho (t, \cdot)  \quad \text{strongly in  $L^1_{\mathrm{loc}} (\R)$, for every $t>0$.}
 \end{equation}
\\
{\bf Step 2} (compactness of $\{ u_\ee \}_{\ee>0})$. Since $\partial_x u_\ee=\rho_\ee z_\ee$, by combining \eqref{e:mpii} and \eqref{e:boundrho} we obtain that $\|\partial_x u_\ee\|_{L^\infty (\R^+\times \R}~\le~2 \|z_0\|_{L^\infty(\R^+\times \R)}$. Owing to~\eqref{e:nlGARZ}, this implies that the functions $\rho_\ee$ are equi-Lipschitz functions in $\R^+\times \R$. By the Arzel\`a Ascoli Theorem, up to taking a further subsequence,  we have 
\be \label{e:conveu}
     u_{\ee_k} \to \bar u \;  \text{uniformly in $C^0_\loc(\R^+\times \R)$}, \qquad \partial_t u_{\ee_k}\weaks \partial_t u, \;  \partial_x u_{\ee_k}\weaks  \partial_x \bar u \; \text{weakly$^\ast$ in $L^\infty(\R_+ \times \R)$}. 
     \eq 
By combining~\eqref{e:Linfglobal3} and~\eqref{e:conveu} we can pass to the limit in~\eqref {e:nlGARZee} and conclude that $(\bar \rho, \bar u)$ satisfies the equation at the first line of~\eqref{e:GARZ} in the sense of distributions, and that the equation at the second lines holds true as an identity between $L^\infty$ functions. \\
{\bf Step 3} (entropy admissibility of the limit). We are left to show that $\bar \rho$ is an entropy admissible solution of the conservation laws at the first line of~\eqref{e:GARZ}. Owing to the uniqueness result in~\cite[Theorem 1.1]{MS:localGARZ}, this implies that the whole family $(\rho_\ee, u_\ee)$ converges to $(\bar \rho, \bar u)$ and concludes the proof. 

By combining~\eqref{e:tvxi} with~\eqref{e:Linfglobal2} and using the lower semicontinuity of the total variation with respect to the $L^1$ convergence we obtain that  $\bar \rho \in L^\infty_{\loc}(]0,+\infty[ \, ; \mathrm{BV}_{\loc}(\R))$. Since $\rho$ is distributional solution of \eqref{e:GARZ}, we deduce that $\bar \rho \in \mathrm{BV}_{\loc}(]0,+\infty[\times \R).$ 
Since $\bar u$ is a  Lipschitz continuous function, for every $k \in \R$ the distribution 
\[
\mu_k :=    \partial_t |\bar \rho- k| + \partial_x \Big[\mathrm{sign}[\bar \rho-k] \big[V(\bar \rho, \bar u) \rho - V(k, \bar u) k \big]\Big] +
    \mathrm{sign}[\bar \rho-k] k \partial_2 V (k, u) \partial_x \bar u
\]
is a locally finite Borel measure and by the Volpert chain-rule for bounded variation functions~\cite[\S3.10]{AmbrosioFuscoPallara}, we deduce that $\mu_k$ is concentrated on the jump set $J$ of $\rho$.  More precisely
\[
\mu_k = \left[ \frac{\psi_k(\rho^+, \bar u)-\psi_k(\rho^-,\bar u)}{\rho^+-\rho^-} -\lambda (|\rho^+-k| - |\rho^--k|)\right]\frac1{\sqrt{1+\lambda^2}} \mathcal H^1 \llcorner J,
\]
where $\rho^\pm$ denotes the left and right traces, respectively, of $\rho$ at $J$, 
\[
\lambda = \frac{V(\rho^+,u)\rho^+-V(\rho^-,u)\rho^-}{\rho^+-\rho^-}
\]
is the speed of propagation of the shock dictated by the Rankine-Hugoniot conditions and
\[
\psi_k(\rho,u)= \mathrm{sign}[\rho-k] \big[V(\rho, u) \rho - V(k, u) k \big].
\]
By using~\eqref{e:Linfglobal2} and passing to the limit in the inequality~\eqref{e:onesidedLip} we conclude that $\bar \rho$ satisfies a one-sided Lipschitz condition, which in turn implies 
$\rho^- \le \rho^+$.  Since for every $u \in \R$ the map $\rho\mapsto V(\rho,u)\rho$ is concave, we conclude that 
the density of $\mu_k$ with respect to $\mathcal H^1$ is non-positive. 
This establish the entropy admissibility of $\rho$ and concludes the proof of Theorem~\ref{t:oleinik}.
\end{proof}
\section*{Acknowledgments}
LVS wishes to thank Ganesh Vaidya for drawing her attention to reference~\cite{CFGG}. 
Both authors are members of the GNAMPA group of INDAM and are supported by the project PRIN 2022 PNRR C53D23008420001 (PI Roberta Bianchini). LVS is also supported by PRIN 2022YXWSLR (PI Paolo Antonelli) and by the CNR project STRIVE (DIT.AD022.207 - STRIVE (FOE 2022)). Both PRIN projects are financed by the European Union-Next Generation EU. 
\bibliographystyle{plain}
\bibliography{garz}
\end{document}